\documentclass[12pt, leqno]{article}
\usepackage{amsmath}
\usepackage{amsthm}
\usepackage{amssymb}
\usepackage{latexsym}
\usepackage{graphicx}
\allowdisplaybreaks[1]
\usepackage{amsfonts}
\usepackage{ascmac}
\usepackage{color}
\usepackage{enumerate}

\theoremstyle{definition}
\theoremstyle{plain}

\allowdisplaybreaks[1]

\date{}
\newtheorem{Thm}{Theorem}[section]
\newtheorem{Prop}[Thm]{Proposition}
\newtheorem{Lemma}[Thm]{Lemma}

\newcommand{\bu}{\bar{u}}

\newcommand{\dt}{\Delta t}
\newcommand{\dx}{\Delta x}

\newcommand{\dis}{\displaystyle}

\newcommand{\norm}{\parallel}

\newcommand{\Z}{{\mathbb Z}}

\newcommand{\T}{{\mathbb T}}
\newcommand{\N}{{\mathbb N}}
\newcommand{\R}{{\mathbb R}}

\newcommand{\bv}{\bar{v} }
\newcommand{\buc}{\bar{u}^{(c )} }
\newcommand{\bvc}{\bar{v}^{(c )} }

\def\text#1{\mbox{#1 }}

\title{\bf Selection problems of $\Z^2$-periodic\\
 entropy solutions and viscosity solutions}
\author{Kohei Soga
\footnote{Unit\'e de math\'ematiques pures et
appliqu\'ees, CNRS UMR 5669  \&  \'Ecole Normale Sup\'erieure de
Lyon, 46 all\'ee d'Italie, 69364 Lyon, France. The work was supported by ANR-12-BS01-0020 WKBHJ as a researcher for academic year 2014-2015, hosted by Albert Fathi. \newline \indent Currently, Department of Mathematics, Faculty of Science and Technology, Keio University, 3-14-1 Hiyoshi, Kohoku-ku, Yokohama, 223-8522, Japan (soga@math.keio.ac.jp).}}
\begin{document}
\maketitle
\begin{abstract}
\noindent $\Z^2$-periodic entropy solutions of hyperbolic scalar conservation laws and $\Z^2$-periodic viscosity solutions of Hamilton-Jacobi equations are not unique in general. However, uniqueness holds for viscous scalar conservation laws and viscous Hamilton-Jacobi equations. Ugo Bessi \cite{Bessi} investigated the convergence of approximate $\Z^2$-periodic solutions to an exact one in the process of the vanishing viscosity method, and characterized this physically natural $\Z^2$-periodic solution with the aid of Aubry-Mather theory. In this paper, a similar problem is considered in the process of the finite difference approximation under hyperbolic scaling. We present a selection criterion different from the one in the vanishing viscosity method, which exhibits difference in characteristics between the two approximation techniques.      

\medskip

   \medskip

\end{abstract}
\setcounter{section}{0}
\setcounter{equation}{0}
\section{Introduction}
We consider hyperbolic scalar conservation laws and the corresponding Hamilton-Jacobi equations
\begin{eqnarray}\label{CL}
&&u_t+H(x,t,c+u)_x=0,\\\label{HJ}
&&v_t+H(x,t,c+v_x)=h(c ),
\end{eqnarray}
where $c, h(c )\in\R$ are given constants. Here, the function $H(x,t,p)$ is assumed to satisfy the following (H1)--(H4):

(H1) $H(x,t,p):\T^2\times\R\to\R$, $C^2$, \quad\\
\indent (H2) $H_{pp}>0$,  \quad\\
\indent(H3) $\dis \lim_{|p|\to+\infty}\frac{H(x,t,p)}{|p|}=+\infty$,

\noindent where $\T:=\R/\Z$. $\T$ is identified with $[0,1)$, and a function defined on $\T$ is regarded as a $1$-periodic function defined on $\R$. We say that a function $f(x,t)$ is $\Z^2$-periodic, if it is $1$-periodic in both $x$ and $t$. It follows from (H1)--(H3) that the Legendre transform $L(x,t,\xi)$ of $H(x,t,\cdot)$ is well-defined, and is given by 
$$L(x,t,\xi)=\sup_{p\in\R}\{\xi p -H(x,t,p)\}.$$ 
Note that the function $L$ satisfies

(L1) $L(x,t,\xi):\T^2\times\R\to\R$, $C^2$, \quad\\
\indent (L2) $L_{\xi\xi}>0$,  \quad\\
\indent (L3) $\dis \lim_{|\xi|\to+\infty}\frac{L(x,t,\xi)}{|\xi|}=+\infty$.

\noindent The final assumption for $H$ is: 

 (H4) There exists $\alpha>0$ such that $|L_x|\le\alpha(|L|+1)$.

\noindent  A flux function $H$ satisfying (H1) and (H2) is common in continuum mechanics. A function $H$ satisfying (H1)--(H4) is common in Hamiltonian dynamics, and is called a Tonelli Hamiltonian. Note that (H4) implies completeness of the Euler-Lagrange flow generated by $L$, and hence the Hamiltonian flow generated by $H$. Aubry-Mather theory extensively investigates the Hamiltonian dynamics and Lagrangian dynamics generated by a Tonelli Hamiltonian $H$ and its Legendre transform $L$.  Weak KAM theory unifies Aubry-Mather theory, $\Z^2$-periodic entropy solutions of (\ref{CL}) and $\Z^2$-periodic viscosity solutions of (\ref{HJ}), providing many useful tools for the analysis of the PDEs and dynamical systems \cite{Fathi}, \cite{Fathi-book}, \cite{WE}. 

We briefly explain the background of our selection problem in terms of the large-time behaviors of entropy solutions and viscosity solutions. 
It is well-known that the initial value problems 
\begin{eqnarray}\label{CLi}
\left\{
\begin{array}{lll}
&\dis u_t+H(x,t,c+u)_x=0\,\,\,\,\mbox{in $\T\times(0,T]$,}\medskip\\
&\dis u(x,0)=u^0(x)\in L^\infty(\T)\,\,\,\,\mbox{on $\T$},\quad\int_\T u^0(x)dx=0,
\end{array}
\right.
\end{eqnarray}
\begin{eqnarray}\label{HJi}
\left\{
\begin{array}{lll}
&\dis v_t+H(x,t,c+v_x)=h(c)\,\,\,\,\mbox{in $\T\times(0,T]$,}\medskip\\
&v(x,0)=v^0(x)\in Lip(\T)\,\,\,\,\mbox{on $\T$}, 
\end{array}
\right.
\end{eqnarray}
are uniquely solvable in the sense of entropy solutions and viscosity solutions, where  $u\in C^0((0,T],L^1(\T))$ and $v\in Lip(\T\times(0,T])$, respectively.  In the spatially one-dimensional case, (\ref{CLi}) and (\ref{HJi}) are equivalent in the sense that the entropy solution $u$ or viscosity solution $v$ is derived from the other if $u^0=v^0_x$. In particular, we have  $u=v_x$ (see, e.g., \cite{Bernard}).
From now on we assume that $u^0=v^0_x$, and that an entropy solution $u\in C^0((0,T],L^1(\T))$ means the representative element given by $v_x$. 

The viscosity solution $v$ of (\ref{HJi}) exists for $T\to+\infty$ and tends to a time-periodic viscosity solution of (\ref{HJ}) with a period greater or equal to one as $t\to\infty$ \cite{Bernard-R} (and the references cited there). The periodic state may depend on initial data, namely, time-periodic viscosity solutions with each period are not unique with respect to $c$ in general. Since the entropy solution  $u$ is equal to $v_x$, similar large-time behaviors and the multiplicity of periodic states hold for (\ref{CLi}) and (\ref{CL}) \cite{Bernard}. In particular, $\Z^2$-periodic entropy solutions of (\ref{CL}) and $\Z^2$-periodic viscosity solutions of (\ref{HJ}) are not unique with respect to $c$ in general. Note that for these large-time behaviors, $h(c )$ must be a special value called the ``effective Hamiltonian''. Otherwise $v$ has linear growth in time.   The effective Hamiltonian is $C^1$ in one-dimensional problems.  For details on the effective Hamiltonian, see \cite{Bernard}. From now on, $h(c )$ is assumed to be the effective Hamiltonian.

It is common to approximate the entropy solution $u$ of (\ref{CLi}) and viscosity solution $v$ of (\ref{HJi}) by the smooth solutions of the following parabolic equations with $\nu>0$, the periodic boundary condition and the same initial data as the above, 
\begin{eqnarray}\label{CLiv}
&&u^\nu_t+H(x,t,c+u^\nu)_x=\nu u^\nu_{xx},\\\label{HJiv}
&&v^\nu_t+H(x,t,c+v^\nu_x)=h^\nu(c )+\nu v^\nu_{xx}.
\end{eqnarray}
 Within each bounded time interval, the convergence $u^\nu\to u$ and $v^\nu\to v$ as $\nu\to 0$ can be proved. This is called the  {\it vanishing viscosity method}.  For each initial data, the solutions $u^\nu$ and $v^\nu$ exist for $T\to+\infty$, and tend to time-periodic solutions $\bu^\nu$ of (\ref{CLiv}) and $\bv^\nu$ of (\ref{HJiv}) as $t\to\infty$, respectively. Unlike the inviscid problems, $\bu^\nu$ and $\bv^\nu$ are unique with respect to $c$ and their period is exactly equal to one (to be precise, $\bv^\nu$ is unique up to constants). Note that we need to choose an appropriate constant $h^\nu(c )$ for the asymptotic behavior. $h^\nu(c )$ is also called the effective Hamiltonian, and tends to $h(c )$ as $\nu\to0$. From the family $\{\bu^\nu\}_{\nu>0}$ (resp. $\{\bv^\nu\}_{\nu>0}$ (adding a constant to $\bv^\nu$, if necessary)), we can take a convergent subsequence, whose limit is a $\Z^2$-periodic entropy solution of (\ref{CL}) (resp. $\Z^2$-periodic viscosity solution of (\ref{HJ})) \cite{JKM}, \cite{Bessi}. Since $\Z^2$-periodic entropy solutions and viscosity solutions are not necessarily unique with respect to $c$,  an interesting problem arises: {\it  Does $\{\bu^\nu\}_{\nu>0}$ (resp. $\{\bv^\nu\}_{\nu>0}$) accumulate on a single $\Z^2$-periodic entropy solution $\bu$ (resp. $\Z^2$-periodic viscosity solution $\bv$) as $\nu\to0$? If this is the case, how can $\bu$ and $\bv$ be characterized?}  This problem is partially solved in \cite{JKM}, \cite{Bessi}. 

Selection problems arise also in higher dimensional stationary problems 
\begin{eqnarray}\label{HJs}  
H(x,t,c+\bv_x)=h(c )\mbox{\,\,\, in $\T^n$},
\end{eqnarray}
where the viscosity solutions are not unique in general with respect to $c\in\R^n$.  The vanishing viscosity method also works with the elliptic equation 
\begin{eqnarray}\label{HJsv}  
H(x,t,c+\bv^\nu_x)=h^\nu(c )+\nu  \Delta\bv^\nu \mbox{\,\,\, in $\T^n$},
\end{eqnarray}
which is uniquely (up to constants) solvable for each $\nu>0$ and $c$. From the family $\{  \bv^\nu\}_{\nu>0}$ (adding a constant to each $\bv^\nu$, if necessary), we can take a convergent subsequence, whose limit is a viscosity solution of (\ref{HJs}). The selection problem in (\ref{HJs}) and (\ref{HJsv}) is partially solved in \cite{Nalini}. There is another approximation method called the {\it ergodic approximation} with the discounted Hamilton-Jacobi equation 
\begin{eqnarray}\label{disco}  
\varepsilon v^\varepsilon+H(x,t,c+v^\varepsilon_x)=h(c ) \mbox{\,\,\, in $\T^n$}. 
\end{eqnarray}
This is also uniquely solvable for each $\varepsilon>0$ and $c$. From the family $\{  \bv^\varepsilon\}_{\varepsilon>0}$, we can take a convergent subsequence, whose limit is a viscosity solution of (\ref{HJs}).  The selection problem in (\ref{HJs}) and (\ref{disco}) is partially solved in \cite{Iturriaga} and then almost completely in  \cite{Davini}, based on weak KAM theory. Recently, another approach to this selection problem is announced in \cite{MT} based on the nonlinear adjoint method, which covers degenerate viscous Hamilton-Jacobi equations as well. 

The {\it finite difference approximation} is also a common technique to obtain the entropy solution $u$ of (\ref{CLi}) and  viscosity solution $v$ of (\ref{HJi}), which is a simple and realistic approximation available on a computer. In the case of a Tonelli Hamiltonian, however, it is not easy to verify the stability and convergence of finite difference schemes for $T\to+\infty$ and to study large-time behaviors of difference solutions for possible periodic states. Recently, the author announced a new approach to the Lax-Friedrichs finite difference scheme based on probability theory and calculus of variations, and obtained new results on time-global stability, large-time behaviors, error estimates, existence of periodic difference solutions, etc., with lots of useful details in application of the scheme to weak KAM theory \cite{Soga2}, \cite{Soga3}.   These arguments pose a selection problem of $\Z^2$-periodic entropy solutions and viscosity solutions as well. A Lax-Oleinik type operator for the Lax-Friedrichs  scheme, which is introduced in \cite{Soga2}, \cite{Soga3}, is the basic tool for the investigation of the selection problem. We discretize (\ref{CLi}) with the Lax-Friedrichs scheme and  (\ref{HJi}) with a scheme so that the following two difference equations are equivalent:
\begin{eqnarray}\label{CL-Delta}
&& \frac{u^{k+1}_{m+1}-\frac{(u^k_{m}+u^k_{m+2})}{2}}{\dt}
+\frac{H(x_{m+2},t_k,c+u^k_{m+2})-H(x_m,t_k,c+u^k_m)}{2\dx}
=0,\\\label{HJ-Delta}
&&\frac{ v^{k+1}_{m} - \frac{(v^k_{m-1}+v^k_{m+1})}{2} }{\dt}+H(x_{m},t_k,c+\frac{v^k_{m+1}-v^k_{m-1}}{2\dx})
=h_\Delta(c ),
\end{eqnarray}   
We will give details on the discretization in Section 3. It is proved that the two schemes are globally stable with fixed $\Delta=(\dx,\dt)$, if $\lambda:=\dt/\dx<\lambda_1$ with an appropriate number $\lambda_1$ and if $h_\Delta(c )$ is chosen properly. $h_\Delta(c )$ is also called the effective Hamiltonian for the difference Hamilton-Jacobi equation, and tends to $h(c )$ as $\Delta\to0$. Furthermore,  any solutions $u^k_m$ and $v^k_{m+1}$, tend to periodic difference solutions $\bu^k_m$ of (\ref{CL-Delta}) and $\bv^k_{m+1}$ of (\ref{HJ-Delta}) as the time-index $k\to+\infty$, respectively. $\bu^k_m$ and $\bv^k_{m+1}$ are unique (up to constant for $\bv^k_{m+1}$) with respect to $c$, and have the period one. 
Let $\bu_\Delta$ be the step function derived from $\bu^k_m$, and let $\bv_\Delta$ be the linear interpolation of $\bv^k_m$. Then, from the family $\{ \bu_\Delta \}_{\dx>0,\dt=\lambda \dx}$ (resp. $\{ \bv_\Delta \}_{\dx>0,\dt=\lambda \dx}$ (adding a constant to each $\bv_\Delta$, if necessary)) with {\it hyperbolic scaling}:  $0<\lambda_0\le\lambda<\lambda_1$ ($\lambda$ is fixed), we can take a convergent subsequence, whose limit is a $\Z^2$-periodic entropy solution of (\ref{CL}) (resp. $\Z^2$-periodic viscosity solution of (\ref{HJ})). Here is the selection problem:      
\medskip

\noindent{\bf Selection problem.} {\it  Does $\{ \bu_\Delta \}_{\dx>0,\dt=\lambda \dx}$ (resp. $\{ \bv_\Delta \}_{\dx>0,\dt=\lambda \dx}$)  accumulate on a single $\Z^2$-periodic entropy solution $\bu$ (resp. $\Z^2$-periodic viscosity solution $\bv$) as $\Delta\to0$ under hyperbolic scaling? If this is the case, how can $\bu$ and $\bv$ be characterized?}  
\medskip

\noindent It follows from \cite{Oleinik} that if we take {\it diffusive scaling}, i.e., $\dx\to0$ with $\dx^2/\dt=O(1)$, then $\bu_\Delta$ and $\bv_\Delta$ tend to the unique $\Z^2$-periodic solutions of (\ref{CLiv}) and (\ref{HJiv}), respectively, and no selection problem arises. 

The purpose of this paper is to present an answer to the  selection problem above, comparing it with the selection problem in the vanishing viscosity method \cite{Bessi}. The contributions of this work are: 
\begin{itemize}
\item[(1)] This is the first attempt to formulate and solve the selection problem of $\Z^2$-periodic entropy solutions and viscosity solutions in the  finite difference approximation, 
\item[(2)]Mathematically and physically new aspects of the finite difference approximation are made clear in contrast with the vanishing viscosity method, particularly from the viewpoints of scaling limit of random walks,
 \item[(3)] Our result leads to better understanding of various numerical experiments.   
\end{itemize}
\setcounter{section}{1}
\setcounter{equation}{0}
\section{Result}
Throughout this paper, $\bu$ denotes a $\Z^2$-periodic entropy solution of (\ref{CL}), $\bv$ a $\Z^2$-periodic viscosity solution of (\ref{HJ}), $\bu^k_m$ a $\Z^2$-periodic difference solution of (\ref{CL-Delta}), $\bv^k_{m+1}$ a $\Z^2$-periodic difference solution of (\ref{HJ-Delta}), $\bu_\Delta$ the step interpolation of $\bu^k_m$ and $\bv_\Delta$ the linear interpolation of $\bv^k_{m+1}$. In order to specify the value $c$, the notation $\buc$, $\bvc$, $\buc_\Delta$, $\bvc_\Delta$, $\bu^k_m(c )$, $\bv^k_{m+1}(c )$, etc., is sometimes used. Define quotient  maps 
\begin{eqnarray*}
&&pr: \R\ni x\mapsto x\hspace{-3mm}\mod1\in\T,\\
&&pr:\R^2\ni (x,s)\mapsto(x\hspace{-3mm}\mod1,s\hspace{-3mm}\mod1)\in\T^2,\\
&&pr:\R^3\ni (x,s,y)\mapsto(x\hspace{-3mm}\mod1,s\hspace{-3mm}\mod1,y)\in\T^2\times\R.
\end{eqnarray*} 
In this paper, we use the term ``projection'' for the operation of $pr$ (it does not mean ``$\T\times\R\ni(x,y)\mapsto x\in\T$''). 

 According to weak KAM theory, we have a characteristic curve $\gamma:(-\infty,t]\to\R$ of $\bvc$ or $\buc=\bvc_x$ such that 
 $$\mbox{$\gamma(t)=x$\,\,\, and\,\,\, $\gamma'(s)=H_p(\gamma(s),s,c+\buc(\gamma(s),s))$, $s<t$}$$
 for each $x,t\in\R$. Furthermore, $\gamma$ falls into a global characteristic curve $\gamma^\ast:\R\to\R$ as $s\to-\infty$ (otherwise $\gamma$ is a part of a global characteristic curve). The curve $\gamma^\ast$ has the rotation number 
 $$\lim_{|s|\to\infty}\frac{\gamma^\ast(s)}{s}=h'(c ).$$
  For each $c$, the set 
$$\mathcal{M}^{(c )}:=\bigcup_{\gamma^\ast}\{ pr(\gamma^\ast(s), s)\,\,|\,\,s\in\R\}\subset \T^2$$ 
is called the Aubry-Mather set for $c$, where the union is taken over all the global characteristic curves $\gamma^\ast$ of all $\Z^2$-periodic entropy solutions or viscosity solutions with $c$. It is known that, if the rotation number $h'(c )$ is irrational, $\buc$ and $\bvc$ are unique (up to constants for $\bvc$) \cite{WE}. Hence, we consider the case where $h'(c )$ is rational. In particular, we deal with the situation where the  global characteristic curves yielding $\mathcal{M}^{(c )}$ form smooth hyperbolic stable/unstable manifolds, and the graph of each $\buc$, 
$$\mbox{graph$(c+\buc):=\{ (x,t,c+\buc(x,t))\,|\,x,t\in\T\}$},$$
consists of parts of these stable/unstable manifolds. It is easy to see an example of such a situation with non-uniqueness of $\bu$ and $\bv$ through explicitly solvable problems given by Hamiltonians of the form $H(x,t,p)=\frac{1}{2}p^2-F(x)$.  Here are the assumptions for our selection problem: Throughout this paper, $\lambda=\dt/\dx$ is fixed with $0<\lambda_0\le\lambda<\lambda_1$, where $\lambda_1$ is from the stability analysis of (\ref{CL-Delta}) and (\ref{HJ-Delta}) and $\lambda_0$ is arbitrary. 
\begin{itemize}{\it 
\item[(A1)] For all $c\in[c_0,c_1]$, the Aubry-Mather set $\mathcal{M}^{(c )}$ consists of common hyperbolic periodic orbits $\gamma^\ast_1, \ldots,\gamma^\ast_I$ with a rational rotation number $p/q$, $q\in\N,p\in\Z$, where $0\le \gamma^\ast_1(0)<\gamma^\ast_2(0)<\cdots<\gamma^\ast_I(0)<1$ and $\{pr(\gamma^\ast_i(s),s)\,|\,s\in\R\}\neq \{pr(\gamma^\ast_j(s),s)\,|\,s\in\R\}$ for $i\neq j$.  
\item[(A2)] If $c=c_0$ (resp. $c_1$), graph$(c+\buc)$ coincides with the lower separatrix (resp. the upper separatrix) formed by the projected $C^1$-stable/unstable manifolds of $\gamma^\ast_1,\ldots, \gamma^\ast_I$. 
\item[(A3)] Let $\mathcal{U}^i_{loc}(\delta)$ denote the unstable manifold of $\gamma^\ast_i$ restricted to 
$$U^i(\delta):=\{ (x,t)\in\R^2\,|\, |x-\gamma^\ast_i(t)|\le\delta\}.$$
Let $\zeta^i(x,t)$ be a $C^2$-generating function of the local unstable manifold $\mathcal{U}^i_{loc}(\delta)$, i.e., $\{(x,t,\zeta^i_x(x,t))\,|\,(x,t)\in U^i(\delta)\}=\mathcal{U}^i_{loc}(\delta)$, and set  
$$ \Gamma^i:=\int_0^q\{\zeta^i_{xx}(\gamma^\ast_i(s),s)-\lambda^2\zeta^i_{tt}(\gamma^\ast_i(s),s)\}ds,\,\,\,i=1,\ldots,I.$$
Then, $\min_i \Gamma^i$ is attained by only one $i=i^\ast$.    
\item[(A4)] If the limit $\buc$ of a convergent subsequence $\{\buc_\Delta\}$ satisfies graph$(c+\buc)\supset pr\mathcal{U}^i_{loc}(\delta)$ for small $\delta>0$, then $\buc_\Delta$ is equi-Lipschitz with respect to $x$ on $U^i(\delta)$, i.e.,
$$|\bu^k_{m+2}(c)-\bu^k_m(c)|\le \theta_0\cdot2\dx\mbox{ for all $(x_m,t_k),(x_{m+2},t_k)\in U^i(\delta)$,}$$
with $\theta_0$ independent of $\Delta$, $c$ and the choice of the subsequence. 
}\end{itemize} 
We refer to known results on $\buc_\Delta$ and $\bvc_\Delta$  in Section 3. Before stating our result, we recall the Peierls barriers, which play an important role in Aubry-Mather theory.  For each $c$, the Peierls  barrier $h^{(c )}_p(x,t;y,\tau):\T^2\times\T^2\to\R$ is defined as 
\begin{eqnarray*}
h^{(c )}_p(x,t;y,\tau):=\liminf_{T\in\N,T\to\infty}\left[ \inf \int^{\tau+T}_t \{L(\kappa(s),s,\kappa'(s))-c\kappa'(s)+h(c )\}ds   \right],
\end{eqnarray*}
where the infimum is taken over all absolutely continuous curves $\kappa:[t,\tau+T]\to\T$ with $\kappa(\tau+T)=y$ and $\kappa(t)=x$. For details on the Peierls barrier, see \cite{Mather2}, \cite{Fathi-book}. Our main result is: 
\begin{Thm}\label{main}
Suppose that (A1)--(A4) hold and that $c\in(c_0,c_1)$. Add a constant to $\bvc_\Delta$ so that $\bar{v}^{(c )}_\Delta(\gamma^\ast_{i^\ast}(0),0)=0$. Then, as $\Delta\to0$ with $0<\lambda_0\le\dt/\dx=\lambda<\lambda_1$, 
\begin{enumerate}
\item[(1)] $\bar{v}^{(c )}_\Delta$ converges to $h^{(c )}_p(\gamma^\ast_{i^\ast}(0),0;\cdot,\cdot)$ uniformly,
\item[(2)] $\bar{u}^{(c )}_\Delta$ converges to $\bar{u}^{(c )}=(h^{(c )}_p(\gamma^\ast_{i^\ast}(0),0;\cdot,\cdot))_x$ pointwise a.e. and in the $C^0(\T;L^1(\T))$-norm. A geometrical characterization of $\buc$ is that there exists $\delta>0$ such that graph$(c+\bar{u}^{(c )})$ contains $pr\mathcal{U}^i_{loc}(\delta)$ only for $i=i^\ast$.
\end{enumerate}
\end{Thm}
We compare our result with the one in the vanishing viscosity method \cite{Bessi}. Our setting given by (A1) and (A2) is essentially the same as the one in \cite{Bessi}.  In both of the cases, the basic ingredients are stochastic Lax-Oleinik type operators for (\ref{HJiv}) and (\ref{HJ-Delta}), and estimates for the rate of the law of large numbers realized in the operators as $\nu\to0$ and $\Delta\to0$. In \cite{Bessi}, the standard framework of the stochastic Lax-Oleinik type operator for viscous Hamilton-Jacobi equations \cite{Fleming} is used. We will see our stochastic Lax-Oleinik type operator for (\ref{HJ-Delta}) in Section 3.  Although the mechanisms of the selection are similar to each other, the selection criteria are different. Roughly speaking, our case also relies on the so-called numerical viscosity of the Lax-Friedrichs scheme, which causes diffusive effect similar to that of the artificial viscosity term $\nu v^\nu_{xx}$. However, the speed of the diffusion caused by the numerical viscosity is $\dx/\dt=\lambda^{-1}$, and due to hyperbolic scaling, it is finite for $\Delta\to0$. This leads to the different selection criterion.  In fact, (A3), which gives our selection criterion, is different from the one in \cite{Bessi}. For the vanishing viscosity method, the values 
$$\tilde{\Gamma}^i:=\int_0^q\zeta^i_{xx}(\gamma^\ast_i(s),s)ds,\quad i=1,\ldots,I$$
appear, and $\bv^\nu(c )$ selects $h^{(c )}_p(\gamma^\ast_{\tilde{i}^\ast}(0),0;\cdot,\cdot)$ with $\tilde{i}^\ast$ that minimizes $\tilde{\Gamma}^i$. In our case, the selection criterion contains the discretization  parameter $\lambda(\ge\lambda_0>0)$ and we may change the selection by varying $\lambda$, which means that our finite difference method with hyperbolic scaling is more ``flexible'' or ``closer'' to the exact hyperbolic problem than the vanishing viscosity method. Our method does not always select the physically natural solutions selected by the vanishing viscosity method. If we choose $\lambda>0$ small enough, then our selection is the same as the one in the vanishing viscosity method. 	 
This is an interesting difference between the two methods, since  they  are often regarded as mathematically similar due to their diffusive effects. 
At the current stage, the regularity assumption (A4) plays an essential role yielding an appropriate estimate for the rate of the law of large numbers, whereas there is no such assumption in \cite{Bessi}.  This is due to different structures of the stochastic Lax-Oleinik type operators. In the case of the vanishing viscosity method, it is enough to deal with the standard Brownian motions. In our case, we have to investigate continuous limits of space-time inhomogeneous random walks with hyperbolic scaling. It is not easy to prove the law of large numbers and to estimate its rate for such random walks \cite{Soga1}. As far as the author knows,  justification of (A4) is an open question, though computer simulations imply that it would be true. Particularly, in the case of entropy solutions with finite number of shocks, their Lipschitzian parts seem to be approximated in equi-Lipschitzian ways (see, e.g., \cite{Nishida-Soga}).   
\setcounter{section}{2}
\setcounter{equation}{0}
\section{Preliminaries}
We state several known facts on $\Z^2$-periodic entropy solutions, $\Z^2$-periodic viscosity solutions and our difference schemes, as well as space-time inhomogeneous random walks arising in the difference equations. 
\subsection{$\Z^2$-periodic viscosity solution}
Let $\bv=\bvc$ be a $\Z^2$-periodic viscosity solution of (\ref{HJ}) which is periodically extended to $\R^2$. Then, $\bv$ satisfies the following equality for each $x,t\in\R$ and $\tau<t$ (the deterministic Lax-Oleinik type operator): 
\begin{eqnarray}\label{value}
\,\,\,\bv(x,t)=\inf_{\gamma\in AC,\,\,\gamma(t)=x}\left\{ \int^t_\tau L^{(c)}(\gamma(s),s,\gamma'(s))ds+\bv(\gamma(\tau),\tau) \right\}+h(c )(t-\tau),    
\end{eqnarray}
where  $AC$ is the family of all absolutely continuous curves $\gamma:[\tau,t]\to\R$ and
$$L^{(c)}(x,t,\xi):=L(x,t,\xi)-c\xi$$
is the Legendre transform of $H(x,t,c+\cdot)$. There is a minimizing curve $\gamma^\ast$, which is a $C^2$-backward characteristic curve of $\bv$, and solves the Euler-Lagrange equation generated by $L^{(c )}$. Moreover, $\bv$ is differentiable with respect to $x$ on the minimizing curve, satisfying the relation
\begin{eqnarray}\label{back}
\gamma^\ast{}'(s)=H_p(\gamma^\ast(s),s,c+\bv_x(\gamma^\ast(s),s))\quad\mbox{for $\tau\le s<t$}.
\end{eqnarray}
Note that,  if $\bv_x(x,t)$ exists, $s=t$ is included in (\ref{back}) and $\gamma^\ast$ is the unique minimizing curve for (\ref{value}). 
Each $\Z^2$-periodic viscosity solution $\bvc$ belongs to $Lip(\T^2)$, and each $\Z^2$-periodic entropy solution $\buc$ belongs to $Lip(\T;L^1(\T))$. We have $\buc=\bvc_x$ and 
\begin{eqnarray}\label{repre}
\bvc(x,t)&=&\bar{U}^{(c )}(x,t)-\int^t_0\int^1_0\{\bar{U}^{(c )}_t(x,\tau)+H(x,\tau,c+\bar{U}^{(c )}_x(x,\tau))\}dxd\tau\\\nonumber
&&+h(c )t,
\end{eqnarray} 
where $\bar{U}^{(c )}(x,t):=\int_0^x\buc(y,t)dy$. For more details, see \cite{Cannarsa}, \cite{Bernard}, \cite{Fathi-book}, \cite{WE}. 
\subsection{Discretization}
Here are some details of our discretization and results on the difference equations obtained in \cite{Soga2}, \cite{Soga3}.  Let $K,N$ be natural numbers with $N\le K$. The mesh size $\Delta=(\dx,\dt)$ is defined by $\dx:=(2N)^{-1}$ and $\dt:=(2K)^{-1}$. Set $\lambda:=\dt/\dx$, $x_m:=m\dx$ for $m\in\Z$ and $t_k:=k\dt$ for $k=0,1,2,\ldots$. Let $(\dx\Z)\times(\dt\Z_{\ge0})$ be the set of all $(x_m,t_k)$, and let
$$\mathcal{G}_{even}\subset (\dx\Z)\times(\dt\Z_{\ge0})\qquad\mbox{(resp. $\mathcal{G}_{odd}\subset (\dx\Z)\times(\dt\Z_{\ge0})$)}$$
be the set of all $(x_m,t_k)$ with $k=0,1,2,\ldots$ and $m\in\Z$ such that  $m+k$ is even (resp. odd), which is called the even grid (resp. odd grid). For $x\in\R$ and $t>0$, the notation $m(x),k(t)$ denotes the integers $m,k$ for which $x\in[x_m,x_m+2\dx)$ on $\mathcal{G}_{even}$ or $\mathcal{G}_{odd}$ and $t\in[t_k,t_k+\dt)$, respectively. Note that $m(x)$ on $\mathcal{G}_{even}$ and $m(x)$ on $\mathcal{G}_{odd}$ are different for the same $x$. We discretize (\ref{CL}) on $\mathcal{G}_{even}$ by the Lax-Friedrichs scheme as (\ref{CL-Delta}). We also discretize (\ref{HJ}) in $\mathcal{G}_{odd}$ as (\ref{HJ-Delta}). 
We introduce the following notation:
$$D_tw^{k+1}_m:=\frac{w^{k+1}_m-\frac{w^k_{m-1}+w^k_{m+1}}{2}}{\dt},\quad D_xw^{k}_{m+1}:=\frac{w^{k}_{m+1}-w^k_{m-1}}{2\dx}.$$
Note that (\ref{CL-Delta}) and (\ref{HJ-Delta}) are equivalent. In particular, if $v^k_m$ satisfies (\ref{HJ-Delta}), then  $u^k_m:=D_xv^k_{m+1}$ satisfies (\ref{CL-Delta}).  Throughout this paper, we follow the convention that the sum of the super-script and sub-script of variables defined on $\mathcal{G}_{even}$ (resp. $\mathcal{G}_{odd}$) is always even (resp. odd) in the notation $\bu^k_m, \bu^{k}_{m(x)},\bv^k_{m+1}, \bv^k_{m(x)}$, etc. We say that $\bu^k_m$ and $\bv^k_{m+1}$ are $\Z^2$-periodic, if they satisfy 
$$\mbox{$\bu^k_{m\pm 2N}=\bu^{k+2K}_{m}=\bu^k_m$,\,\,\,\,\,\,$\bv^k_{m+1\pm 2N}=\bv^{k+2K}_{m+1}=\bv^k_{m+1}$ \,\,\,\,for all $m,k$.}$$  
Let $\bu_\Delta$ be the step function derived from a $\Z^2$-periodic difference  solution $\bu^k_m$, i.e.,
$$\mbox{$\bu_\Delta(x,t):=\bu^k_m$\,\,\,\, for $(x,t)\in[x_{m-1},x_{m+1})\times[t_k,t_{k+1})$.}$$
Let $\bv_\Delta$ be the linear interpolation with respect to the space variable derived from a $\Z^2$-periodic difference solution $\bv^k_{m+1}$, i.e.,
$$\mbox{$\bv_\Delta(x,t):=\bv^k_{m-1}+D_x\bv^k_{m+1}\cdot(x-x_{m-1})$\,\,\,\, for $(x,t)\in[x_{m-1},x_{m+1})\times[t_k,t_{k+1})$.}$$
Note that $\bv_\Delta(x,\cdot)$ is a step function for each fixed $x$ and that $(\bv_\Delta)_x=\bu_\Delta$. 

We sometimes use a phrase like ``a convergent subsequence $\{ \bv_\Delta\}$ which tends to $\bv$ as $\Delta\to0$'', meaning ``a convergent sequence $\{\bv_{\Delta_j}\}_{j\in\N}$ with $\Delta_j\to0$ which tends to $\bv$ as $j\to\infty$''.  
\begin{Prop}[\cite{Soga3}]\label{preliminary} Let $[c_0,c_1]$ be an arbitrary interval. There exist $\lambda_1>0$ and $\delta_1>0$ such that for any $\Delta=(\dx,\dt)$ with $\dx<\delta_1$ and $\lambda=\dt/\dx<\lambda_1$  we have the following statements for each $c\in [c_0,c_1]${\rm:} 
\begin{enumerate}
\item[(1)] There exists one and only one  number $h_\Delta(c ) $ for which (\ref{HJ-Delta}) has the unique (up to constants) $\Z^2$-periodic difference solution $\bv^k_{m+1}$. 
\item[(2)] There exists the unique  $\Z^2$-periodic difference solution $\bu^k_{m}=D_x\bv^k_{m+1}$ of (\ref{CL-Delta}), which is uniformly bounded with respect to $m,k$, $\Delta$ and $c$ with the entropy condition 
$$\frac{\bu^k_{m+2}-\bu^k_m}{2\dx}\le M\mbox{\quad for all $k,m$},$$ 
 where $M>0$ is independent of $\Delta$ and $c$. In particular, the stability condition (CFL-condition) is verified: 
$$|H_p(x_m,t_k,c+\bu^k_m)|\le \lambda_1^{-1}<\lambda^{-1}.$$ 
If $c<\tilde{c}$, then $c+\bu^k_m(c )<\tilde{c }+\bu^k_m(\tilde{c})$ for all $k,m$.
 \end{enumerate}
As $\Delta\to0$ under hyperbolic scaling, i.e., $0<\lambda_0\le\lambda<\lambda_1$,
 \begin{enumerate}
 \item[(3)] $h_\Delta(c )$ converges to the exact effective Hamiltonian $h(c )$  uniformly on $[c_0,c_1]$ with the order $\sqrt{\dx}$.  
\item[(4)] Adding a constant to each $\bv_\Delta$ if necessary, we can subtract a convergent subsequence, which tends to a function $\bv$ uniformly. $\bv$ is a $\Z^2$-periodic viscosity solution of (\ref{HJ}). 
\item[(5)]  Let $\{\bv_\Delta\}$ be the convergent subsequence in (4). Then $\bu_\Delta=(\bv_\Delta)_x$ converges to the $\Z^2$-periodic entropy solution $\bu=\bv_x$ pointwise except for the points of non-differentiability of $\bv$ and also in the $C^0(\T;L^1(\T))$-norm. In particular, $\bu_\Delta$ converges to $\bu$ uniformly in the outside of an arbitrary neighborhood of shocks.  
\end{enumerate}
\end{Prop}
\noindent We remark that $\bu$ and $\bv$ are obtained as a result of large-time behaviors of any other difference solutions \cite{Soga3}. 

 We show the stochastic and variational structure of (\ref{HJ-Delta}). First, we introduce space-time inhomogeneous random walks in $\mathcal{G}_{odd}$, which play the role of  ``characteristic curves'' for (\ref{CL-Delta}) and (\ref{HJ-Delta}).
For each point  $(x_n,t_{l+1})\in\mathcal{G}_{odd}$, we consider backward random walks $\gamma$ that start from $x_n$ at $t_{l+1}$ and move by $\pm\dx$ in each backward time step $\dt$:
$$\gamma=\{\gamma^k\}_{k=l',l'+1,\ldots,l+1},\quad\gamma^{l+1}=x_{n},\quad \gamma^{k+1}-\gamma^k=\pm\dx.$$
More precisely,  for each $(x_n,t_{l+1})\in\mathcal{G}_{odd}$ we introduce the following objects:
\begin{eqnarray*}
&&X^k:=\{ x_{m+1} \,|\, \mbox{ $(x_{m+1},t_k)\in\mathcal{G}_{odd}$, $|x_{m+1}-x_n|\le(l+1-k)\dx$}\}\mbox{ for }k\le l+1,\\
&&G:=\bigcup_{l'< k\le l+1}\big(X^{k}\times\{t_{k}\}\big)\subset\mathcal{G}_{odd}, \\
&&\xi:G\ni(x_{m+1},t_k)\mapsto\xi^k_{m+1}\in[-\lambda^{-1},\lambda^{-1}],\quad \lambda=\dt/\dx, \\
&&\bar{\bar{\rho}}: G\ni(x_{m+1},t_k)\mapsto\bar{\bar{\rho}}^k_{m+1}:=\frac{1}{2}-\frac{1}{2}\lambda\xi^k_{m+1}\in[0,1],\\
&&\bar{\rho}: G\ni(x_{m+1},t_k)\mapsto\bar{\rho}^k_{m+1}:=\frac{1}{2}+\frac{1}{2}\lambda\xi^k_{m+1}\in[0,1],\\
&&\gamma:\{ l',l'+1,\ldots,l+1\}\ni k\mapsto \gamma^k\in X^k,\mbox{ $\gamma^{l+1}=x_n$, $\gamma^{k+1}-\gamma^k=\pm\dx$},\\
&&\Omega:\mbox{the family of the above $\gamma$}.
\end{eqnarray*}
The value $\bar{\bar{\rho}}^k_{m+1}$ (resp. $\bar{\rho}^k_{m+1}$) is regarded as the probability of transition from $(x_{m+1},t_k)$ to $(x_{m+1}+\dx,t_k-\dt)$ (resp. from $(x_{m+1},t_k)$  to $(x_{m+1}-\dx,t_k-\dt)$). The function $\xi$ is a control for random walks, which plays the role of a velocity field on the grid.
We define the density of each path $\gamma\in\Omega$ as
$$\mu(\gamma):=\prod_{l'< k\le l+1}\rho(\gamma^{k},\gamma^{k-1}),$$
where $\rho(\gamma^{k},\gamma^{k-1})=\bar{\bar{\rho}}^k_{m(\gamma^{k})}$ (resp. $\bar{\rho}^k_{m(\gamma^{k})}$) if $\gamma^{k}-\gamma^{k-1}=-\dx$ (resp. $\dx$).
The density $\mu(\cdot)=\mu(\cdot;\xi)$ yields a probability measure for $\Omega$, i.e.,
$$prob(A)=\sum_{\gamma\in A}\mu(\gamma;\xi)\mbox{\quad for $A\subset\Omega$}. $$
The expectation with respect to this probability measure is denoted by $E_{\mu(\cdot;\xi)}$, namely, for a random variable $f:\Omega\to\R$ we have
$$E_{\mu(\cdot;\xi)}[f(\gamma)]=\sum_{\gamma\in\Omega}\mu(\gamma;\xi)f(\gamma).$$
We use $\gamma$ as the symbol for random walks or a sample path. If necessary, we write $\gamma=\gamma(x_n,t_{l+1};\xi)$ in order to specify its initial point and control.

\indent We state several important results on the scaling limit of such inhomogeneous random walks, obtained in \cite{Soga1}. Note that the variance $\sigma^k:=E_{\mu(\cdot;\xi)}[|\gamma^k-\bar{\gamma}^k|^2]$ with $\bar{\gamma}^k:=E_{\mu(\cdot;\xi)}[\gamma^k]$ is of the order $O(1)$ for inhomogeneous random walks in general, whereas it is of the order $O(\dx)$ for space homogeneous cases (see Remak 3.3 of \cite{Soga1}). Let $\eta(\gamma)=\{\eta^k(\gamma)\}_{k=l',l'+1,\ldots,l+1}$, $\gamma\in\Omega$ be a random variable that is  induced by a random walk $\gamma=\gamma(x_n,t_{l+1};\xi)$ as
\begin{eqnarray}\label{etaeta}
\eta^{l+1}:=\gamma^{l+1},\,\,\,\,\,\, \eta^k(\gamma):=\gamma^{l+1}-\sum_{k< k'\le l+1}\xi(\gamma^{k'},t_{k'})\dt\mbox{ \,\,\, for $l'\le k\le l$}.
\end{eqnarray}
Set $\tilde{\sigma}^k:=E_{\mu(\cdot;\xi)}[|\gamma^k-\eta^k(\gamma)|^2]$ and $ \tilde{d}^k:=E_{\mu(\cdot;\xi)}[|\gamma^k-\eta^k(\gamma)|]$ for $l'\le k\le l+1$.
\begin{Prop}[\cite{Soga1}]\label{eta} 
(1) For any control $\xi$, we have 
$$(\tilde{d}^k)^2\le\tilde{\sigma}^k\le \frac{t_{l+1}-t_k}{\lambda}\dx.$$

(2) If a control $\xi$ satisfies the Lipschitz condition at $\bar{\gamma}$: $|\xi^k_{m+1}-\xi^k_\ast|\le\theta |x_{m+1}-\bar{\gamma}^k|$ with $\xi^k_\ast:=\xi^k_{m(\bar{\gamma}^k)}+\frac{\xi^k_{m(\bar{\gamma}^k)+2}-\xi^k_{m(\bar{\gamma}^k)}}{2\dx}(\bar{\gamma}^k-x_{m(\bar{\gamma}^k)})$ for all $k,m$, then we have 
$$\sigma^k\le \frac{e^{4\theta (t_{l+1}-t_k)}}{4\theta\lambda}\dx.$$
\end{Prop}
\noindent If we take the hyperbolic scaling limit $\Delta=(\dx,\dt)\to0$ with 
$0<\lambda_0\le\lambda=\dt/\dx<\lambda_1$, then $(\tilde{d}^k)^2$ and $\tilde{\sigma}^k$ always tend to zero with $O(\dx)$, and so does $\sigma^k$ with the above Lipschitz condition. (A4) will be used to verify the Lipschitz condition.  

Let $\bv^k_{m+1}$ be a $\Z^2$-periodic difference solution of (\ref{HJ-Delta}) that is periodically extended on the whole of $\mathcal{G}_{odd}$. Then $\bv^k_{m+1}$ satisfies the following equality for each $n,l$ and $l'<l$  (our stochastic Lax-Oleinik type operator) \cite{Soga2}:
\begin{eqnarray}\label{value-Delta} 
\bv^{l+1}_{n}\hspace{-2mm}&=&\hspace{-2mm}\inf_{\xi} E_{\mu(\cdot;\xi)}\Big[\sum_{l'<k\le l+1}L^{(c)}(\gamma^k,t_{k-1},\xi^k_{m(\gamma^k)})\dt +\bv^{l'}_{m(\gamma^{l'})}\Big]
+h_\Delta(c )(t_{l+1}-t_{l'}),
\end{eqnarray}
where the infimum is taken over all controls bounded by $\lambda^{-1}$. 
We can find the unique minimizing control $\xi^\ast$. This satisfies with $\lambda_1$ in Proposition \ref{preliminary},
$$|\xi^\ast{}^{k+1}_m|\le\lambda_1^{-1}<\lambda^{-1}\mbox{ \,\,\, and \,\,\, $\xi^\ast{}^{k+1}_{m}=H_p(x_m,t_k,c+D_x \bv^k_{m+1})$}.$$
The equality (\ref{value-Delta}) is the key tool for our proof of Theorem \ref{main}.
\subsection{Construction of $\Z^2$-periodic solution}
We observe how to construct $\Z^2$-periodic entropy solutions and  viscosity solutions with the given position and number of singularities under the assumptions (A1) and (A2). This is necessary  in Section 4.  For  $k=0,\ldots,q-1$, define 
\begin{eqnarray*}
\Theta_{i,k}&:=&pr[\gamma^\ast_i(k),\gamma^\ast_{i+1}(k)],\quad i=1,\ldots,I-1,\\
 \Theta_{I,k}&:=&pr[\gamma^\ast_I(k),\gamma^\ast_{1}(k+1)],\mbox{ if $q>1$},\\
 \Theta_{I,0}&:=&pr[\gamma^\ast_I(0),\gamma^\ast_{1}(0)+1],\mbox{ if $q=1$}. 
\end{eqnarray*}
Note that $\cup_{i,k}\Theta_{i,k}=\T$. Let $C^+_{i,k}$ (resp. $C^-_{i,k}$) be the segment of the upper (resp. lower) separatrix restricted to $\Theta_{i,k}\times\{t=0\}$. For each $i$, choose arbitrarily either $\{ C^+_{i,0},\ldots,C^+_{i,q-1}\}$ or $\{ C^-_{i,0},\ldots,C^-_{i,q-1}\}$. Then, our choice yields  a curve $C=\{(x,U_0(x))\,|\,x\in\T\}$ that consists of $C^+_{i,k}$ or $C^-_{i,k}$ on each $\Theta_{i,k}$. Let $c:=\int_0^1U_0(x)dx$ and let $u_0:=U_0-c$. Then, the solution of (\ref{CLi}) with this $c$ and $u_0$ is a $\Z^2$-periodic entropy solution without any shock.  

Suppose that $\{ C^+_{i,0},\ldots,C^+_{i,q-1}\}$ was chosen in the above $C$ for some $i$. Fix arbitrarily $k_0\in\{0,\ldots,q-1\}$ and take $y_0\in \Theta_{i,k_0}$. Let $\tilde{C}=\{(x,\tilde{U}_0(x))\,|\,x\in\T\}$ be the curve such that $\tilde{C}=C$ on $\T\setminus \Theta_{i,k_0}$ and $\tilde{C}$ switches from $C^+_{i,k_0}$ to $C^-_{i,k_0}$ at $y_0$ as $x$ increases in $\Theta_{i,k_0}$.   Let $\tilde{c}:=\int_0^1\tilde{U}_0(x)dx$ and let $\tilde{u}_0:=\tilde{U}_0-\tilde{c}$, where $\Delta c:=c-\tilde{c}\to0$ as $y_0$ tends to the right boundary of $\Theta_{i,k_0}$. Then the solution  $\tilde{u}$ of (\ref{CLi})$|_{u_0=\tilde{u}_0,c=\tilde{c}}$ is a $q$-periodic entropy solution with a single shock. In fact,  $\tilde{u}(\cdot,1)$ has only one jump at $y_1\in \Theta_{i,k_0+1}$; $\tilde{u}(\cdot,2)$ has only one jump at $y_2\in \Theta_{i,k_0+2}$; $\ldots$ ; $\tilde{u}(\cdot,q)$ has only one jump at $y_q\in \Theta_{i,k_0+q}=\Theta_{i,k_0}$; $y_q=y_0$ due to the conservation law $\int_0^1\{\tilde{c}+\tilde{u}(x,t)\}dx\equiv \tilde{c}$ for all $t$.     

With the above notation, consider the curve $\bar{C}=\{(x,\bar{U}_0(x))\,|\,x\in\T\}$ such that $\bar{C}=C$ on $\T\setminus \cup_{0\le k<q}\Theta_{i,k}$ and $\bar{C}$ switches from $C^+_{i,k_0+k}$ to $C^-_{i,k_0+k}$ at $y_{k}$ as $x$ increases in $\Theta_{i,k_0+k}$ for $k=0,\ldots,q-1$, where $k_0+k$ is replaced by $k_0+k-q$, if $k_0+k\ge q$. Let $\bar{c}:=\int_0^1\bar{U}_0(x)dx=c-q\Delta c$ and let $\bar{u}_0:=\bar{U}_0-\bar{c}$. Then, the solution  $\bar{u}$ of (\ref{CLi})$|_{u_0=\bar{u}_0,c=\bar{c}}$ is a $1$-periodic entropy solution, i.e., a $\Z^2$-periodic entropy solution.   

In this way, we can construct a $\Z^2$-periodic entropy solution  for each $c$, choosing the position and number of shocks, and therefore the corresponding viscosity solution via (\ref{repre}).  
\setcounter{section}{3}
\setcounter{equation}{0}
\section{Proof of result}
First, we briefly state our strategy. The key point is that our discretization scheme is of the first order, namely, for a $C^2$-solution $v$ of (\ref{HJ}) such that $|H_p(z,t,c+v_x(x,t))|\le\lambda_1^{-1}<\lambda^{-1}$ and for $v^k_{m+1}:=v(x_{m+1},t_k)$ defined on $\mathcal{G}_{odd}$, we have  
\begin{eqnarray*}
&&D_tv^{k+1}_m+H(x_m,t_k,c+D_xv^k_{m+1})=h(c )-\frac{\dx}{2\lambda}A(x_m,t_k)+\varepsilon(x_m,t_k),\\
&&A(x,t)=v_{xx}(x,t)-\lambda^2 v_{tt}(x,t),\,\,\,\varepsilon(x_m,t_k)=o(\dx).
\end{eqnarray*}
In fact, this follows from the following Taylor expansions around $(x_m,t_k)$ 
\begin{eqnarray*}
D_tv^{k+1}_m=v_t(x_m,t_k)-\frac{\dx}{2\lambda}A(x_m,t_k)+o(\dx),\,\,\,
D_xv^k_{m+1}=v_x(x_m,t_k)+o(\dx).
\end{eqnarray*}
It follows from a similar reasoning to obtain (\ref{value-Delta}) (see Proposition 2.2 of \cite{Soga2}) that $v^k_{m+1}$ satisfies for any $n,l,l'<l$,
\begin{eqnarray}\label{value-Delta2}
v^{l+1}_n&=&\inf_{\xi}E_{\mu(\cdot;\xi)}\Big[\sum_{l'<k\le l+1}\{L^{(c)}(\gamma^k,t_{k-1},\xi^k_{m(\gamma^k)})-\frac{\dx}{2\lambda}A(\gamma^k,t_{k-1})\\\nonumber
&&\qquad\qquad\qquad\qquad\qquad+\varepsilon(\gamma^k,t_{k-1}) \}\dt +v^{l'}_{m(\gamma^{l'})}\Big]+h(c)(t_{l+1}-t_{l'}).
\end{eqnarray}
By comparing (\ref{value-Delta}) and (\ref{value-Delta2}) near $\gamma^\ast_i$,  we obtain upper and lower estimates for $h(c )-h_\Delta(c )$ in terms of $A$ and so on. Since the law of large numbers holds for the minimizing random walk, where it tends to $\gamma^\ast_i$, we obtain $\Gamma^i$ through $A$ in the estimates. This leads to our criterion. Of course $\Z^2$-periodic viscosity solutions $\bv$ are only Lipschitz, and our argument is more complicated with additional terms. In the estimates for $h(c )-h_\Delta(c )$, the term $\dx\Gamma^i$ appears as well as other terms. In order to keep $\dx\Gamma^i$ as the principle term in the estimates, we need sharper estimates for the others with the aid of Proposition \ref{eta}. The main difficulty is that Chebychev's inequality in our random walks yields estimates of the oder $O(\dx)$ at best -- the same order as $\dx\Gamma^i$. We will refer to this point in the proof.          

We say that a $\Z^2$-periodic entropy solution of $\bar{u}^{(c )}$ or $\Z^2$-periodic viscosity solution $\bvc$ of (\ref{HJ}) has ``transition points'' at $\gamma^\ast_i$, if 
graph$(c+\bu^{(c )})$ or graph$(c+\bvc_x)$ contains $pr\mathcal{U}^i_{loc}(\delta)$ for small $\delta>0$ (``transition'' means that the graph continuously switches from the lower separatrix to the upper one along $\gamma^\ast_i$). The following fact is used in \cite{Bessi}, but for the reader's convenience we repeat it with a proof.
\begin{Prop}\label{peierls}
Let $c\in(c_0,c_1)$, $i\in\{ 1,2,\ldots,I\}$, and $\bv^{(c )}$ be a $\Z^2$-periodic viscosity solution of (\ref{HJ}) with transition points only at $\gamma^\ast_i$. Then, such a $\bvc$ is unique (up to constants) and coincides with $h^{(c )}_p(\gamma^\ast_i(0),0;\cdot,\cdot)$. 
\end{Prop}
\begin{proof}
Let $\bv=\bv^{(c )},\hat{v}=\hat{v}^{(c )}$ be $\Z^2$-periodic viscosity solutions of (\ref{HJ}) with transition points only at $\gamma^\ast_i$. Adding a constant if necessary, we have 
\begin{eqnarray}\label{adjust}
\bv(\gamma^\ast_i(0),0)=\hat{v}(\gamma^\ast_i(0),0)=0.
\end{eqnarray} 
We will prove $\bv=\hat{v}$. 
Take any $(y,\tau)\in \T^2$. Since $\bv$ has transition points only at $\gamma^\ast_i$ and every minimizing curve yields a trajectory on the projected stable/unstable manifolds, we have the following two cases and in both of the cases we obtain $\bv(y,\tau)-\hat{v}(y,\tau)\ge 0$: Let $a\in\Z$ and let $\gamma_0:(-\infty,\tau]\to\R$ be a minimizing curve for $\bv(y,\tau)$.
\medskip

(i) $\{pr(\gamma_0(s),s, c+\bv_x(\gamma_0(s),s))\}_{s<\tau}$ is on the projected unstable manifold of $\gamma^\ast_i$. Then,  there exists $b\in\Z$, $|b|< q$ for which $\gamma_0$ satisfies 
$$\mbox{$pr\gamma_0(aq+b)\to\gamma^\ast_i(0)$ as $a\to-\infty$.}$$
Hence, it follows from (\ref{value}) that $\bv(y,\tau)-\hat{v}(y,\tau)\ge \bv(\gamma_0(aq+b),aq+b)-\hat{v}(\gamma_0(aq+b),aq+b)$ for any $a$. It follows from (\ref{adjust}) that $\bv(\gamma_0(aq+b),aq+b)-\hat{v}(\gamma_0(aq+b),aq+b)\to0$ as $a\to-\infty$. Therefore, we obtain $\bv(y,\tau)-\hat{v}(y,\tau)\ge0$.
\medskip

(ii) Otherwise. Then there exist $x_1, \ldots, x_k\in\{pr\gamma^\ast_j(t)\,|\,t=0,\ldots,q-1, \,\,j\neq i \}$, $k\ge1$ for which we have the heteroclinic chain connecting $y$ and $\gamma^\ast_i(0)$ that consists of the following curves:   
\begin{eqnarray}\nonumber
 \gamma_0&:&(-\infty,\tau]\to\R,\mbox{ the minimizer for $\bv(y,\tau)$, $pr\gamma_0(aq)\to x_1$ as $a\to-\infty$;}\\\nonumber
 \gamma_1&:&\R\to\R,\mbox{ $pr\gamma_1(aq)\to x_2$ (resp. $x_1$) as $a\to-\infty$ (resp. $+\infty$);}\\\label{b}
 &&\qquad\qquad\qquad\qquad \vdots\\\nonumber
 \gamma_{k-1}&:&\R\to\R,\mbox{ $pr\gamma_{k-1}(aq)\to x_k$ (resp. $x_{k-1}$) as $a\to-\infty$ (resp. $+\infty$);}\\\nonumber
 \gamma_k&:&\R\to\R,\mbox{ $pr\gamma_k(aq+b)\to \gamma^\ast_i(0)$ as $a\to-\infty$ with $b\in\Z$, $|b|< q$,}\\\nonumber
 &&\quad\quad\,\,\,\,\,\,\,\,\,\,\, \mbox{$pr\gamma_k(aq)\to x_{k}$ as $a\to+\infty$},
\end{eqnarray}    
with 
$$\mbox{$\gamma'_j(s)=H_p(\gamma_j(s),t,c+\bv_x(\gamma_j(s),s))$ \,\,\,for $j=0,\ldots,k$,}$$
which means that restrictions of $\gamma_0,\ldots,\gamma_k$ within any finite time interval are minimizing curves for $\bv$. Hence, we have 
\begin{eqnarray}\nonumber
 \bv(y,\tau)&=&\int^\tau_{aq}\{L^{(c )}(\gamma_0(s),s,\gamma_0'(s))+h(c)\}ds +\bv(\gamma_0(aq),aq)\\\nonumber
 &=& \lim_{a\to-\infty}\int^\tau_{aq}\{L^{(c )}(\gamma_0(s),s,\gamma_0'(s))+h(c)\}ds +\bv(x_1,0),\\\nonumber
\bv(x_1,0)&=&\lim_{a\to+\infty}\bv(\gamma_1(aq),aq)  \\\nonumber
&=&\lim_{a\to+\infty}\left[\int^{aq}_{-aq}\{L^{(c )}(\gamma_1(s),s,\gamma_1'(s))+h(c)\}ds +\bv(\gamma_1(-aq),-aq)\right]\\\nonumber
&=& \lim_{a\to+\infty}\int^{aq}_{-aq}\{L^{(c )}(\gamma_1(s),s,\gamma_1'(s))+h(c)\}ds +\bv(x_2,0),\\\nonumber
&\vdots&\\\nonumber
\bv(x_k,0)&=&\lim_{a\to+\infty}\bv(\gamma_k(aq),aq)\\\nonumber
&=&\lim_{a\to+\infty}\left[\int^{aq}_{-aq+b}\{L^{(c )}(\gamma_k(s),s,\gamma_k'(s))+h(c)\}ds +\bv(\gamma_k(-aq+b),-aq+b)\right]\\\nonumber
&=&\lim_{a\to+\infty}\int^{aq}_{-aq+b}\{L^{(c )}(\gamma_k(s),s,\gamma_k'(s))+h(c)\}ds +\bv(\gamma^\ast_i(0),0).
\end{eqnarray}
Therefore, we obtain
\begin{eqnarray}\label{kkk}
\bv(y,\tau)&=&\lim_{a\to-\infty}\int^\tau_{aq}\{L^{(c )}(\gamma_0(s),s,\gamma_0'(s))+h(c)\}ds\\ \nonumber
&&+\sum_{j=1}^{k-1} \lim_{a\to+\infty}\int^{aq}_{-aq}\{L^{(c )}(\gamma_j(s),s,\gamma_j'(s))+h(c)\}ds  \\\nonumber
&&+\lim_{a\to+\infty}\int^{aq}_{-aq+b}\{L^{(c )}(\gamma_k(s),s,\gamma_k'(s))+h(c)\}ds +\bv(\gamma^\ast_i(0),0).
\end{eqnarray}
On the other hand, the variational property of $\hat{v}$ implies  
\begin{eqnarray*}
\hat{v}(y,\tau)&\le&\lim_{a\to-\infty}\int^\tau_{aq}\{L^{(c )}(\gamma_0(s),s,\gamma_0'(s))+h(c)\}ds\\ 
&&+\sum_{j=1}^{k-1} \lim_{a\to+\infty}\int^{aq}_{-aq}\{L^{(c )}(\gamma_j(s),s,\gamma_j'(s))+h(c)\}ds  \\
&&+\lim_{a\to+\infty}\int^{aq}_{-aq+b}\{L^{(c )}(\gamma_k(s),s,\gamma_k'(s))+h(c)\}ds +\hat{v}(\gamma^\ast_i(0),0).
\end{eqnarray*}
Thus, we conclude that $\bv(y,\tau)-\hat{v}(y,\tau)\ge \bv(\gamma^\ast_i(0),0)-\hat{v}(\gamma^\ast_i(0),0)=0$.

Since $\hat{v}$ has transition points only at $\gamma^\ast_i$, we have the same argument, obtaining $\bv(y,\tau)-\hat{v}(y,\tau)\le 0$ and $\bv=\hat{v}$. 

Next, we see that $\bv(y,\tau)=h^{(c )}_p(\gamma^\ast_i(0),0;y,\tau)$. Note that 
\begin{eqnarray}\label{Peierle}
h^{(c )}_p(\gamma^\ast_i(0),0;y,\tau)=\liminf_{T\in\N,T\to\infty}\left[ \inf\int^\tau_{-T}\{L^{(c )}(\gamma(s),s,\gamma'(s))+h(c )\}ds \right],
\end{eqnarray}
where the infimum is taken over all absolutely continuous curves $\gamma:[-T,\tau]\to\R$ with $\gamma(\tau)=y$ and $pr\gamma(-T)=\gamma^\ast_i(0)$. Let $\gamma_T:[-T,\tau]\to\R$ be a minimizing curve for the infimum in (\ref{Peierle}). Then, we have for any $T\in\N$,
\begin{eqnarray*}
\bv(y,\tau)&\le& \int^\tau_{-T}\{L^{(c )}(\gamma_T(s),s,\gamma_T'(s))+h(c )\}ds+\bv(\gamma^\ast_i(0),0)\\
&=&\int^\tau_{-T}\{L^{(c )}(\gamma_T(s),s,\gamma_T'(s))+h(c )\}ds.
\end{eqnarray*}
Therefore, we obtain $\bv(y,\tau)\le h^{(c )}_p(\gamma^\ast_i(0),0;y,\tau)$. 

Define $\gamma:[-(2k+1)aq+b,\tau]\to\R$ with $a\in\N$, $\varepsilon>0$, $\gamma_0,\ldots,\gamma_k$ and $b$ in (\ref{b}) as
\[
  \gamma(s) := \left\{ \begin{array}{ll}
    \gamma_0(s) \mbox{\,\,\,\,\,\, for }s\in J_0=[-aq+\varepsilon,\tau], & \\\\
    
    \mbox{linear line connecting $\gamma_0(-aq+\varepsilon)$ and $\gamma_1(aq-\varepsilon)$} & \\
    \mbox{ \qquad\qquad\qquad\quad\quad\,\,\,\,\, for }s\in\tilde{J}_0=[-aq-\varepsilon,-aq+\varepsilon],&\\\\
    
    \gamma_1(s+2aq) \mbox{\,\,\,\,\,\, for }s\in J_1=[-3aq+\varepsilon,-aq-\varepsilon], & \\\\
    
    \mbox{linear line connecting $\gamma_1(-aq+\varepsilon)$ and $\gamma_2(aq-\varepsilon)$,} & \\
     \mbox{ \qquad\qquad\qquad\quad\quad\,\,\,\,\, for }s\in\tilde{J}_1=[-3aq-\varepsilon,-3aq+\varepsilon],&\\
    
  \qquad\qquad\qquad\qquad\qquad  \vdots&\\
    
    \gamma_k(s+2kaq) \mbox{\,\,\,\,\,\, for }s\in J_k=[-(2k+1)aq+b+\varepsilon,-(2k-1)aq-\varepsilon], &\\\\
    \mbox{linear line connecting $\gamma_k(-aq+b+\varepsilon)$
     and $\gamma^\ast_i(-aq)$} &\\
      \mbox{ \qquad\qquad\qquad\qquad\quad\,\,\,\,\, for } s\in\tilde{J}_k=[-(2k+1)aq+b,-(2k+1)aq+b+\varepsilon].&
  
\end{array} \right.
\]
Then, we have
\begin{eqnarray*}
h^{(c )}_p(\gamma^\ast_i(0),0;y,\tau)&\le&\liminf_{a\in\N,a\to+\infty} \int^\tau_{-(2k+1)aq+b}\{L^{(c )}(\gamma(s),s,\gamma'(s))+h(c )\}ds\\
&=&\liminf_{a\in\N,a\to+\infty}\Big[\sum_{j=0}^k\int_{J_j}\{L^{(c )}(\gamma(s),s,\gamma'(s))+h(c )\}ds\\
&&+\sum_{j=0}^k\int_{\tilde{J}_j}\{L^{(c )}(\gamma(s),s,\gamma'(s))+h(c )\}ds\Big].
\end{eqnarray*}
We can take $\varepsilon\to0+$ according to $a\to+\infty$ so that the term on the third line tends to zero. The term on the second line tends to the right-hand side of (\ref{kkk}) as $\varepsilon\to 0+$. Thus, we conclude that $h^{(c )}_p(\gamma^\ast_i(0),0;y,\tau)\le\bv(y,\tau)$, obtaining $\bv(y,\tau)=h^{(c )}_p(\gamma^\ast_i(0),0;y,\tau)$.  \end{proof}

The following proposition is a key fact: 
\begin{Prop}\label{key}
Let $c\in(c_0,c_1)$. If the limit $\bvc$ of a convergent subsequence $\{\bv_\Delta^{(c )}\}$ has transition points at $\gamma_i^\ast$ and $\gamma_j^\ast$, then $\Gamma^i=\Gamma^j$. 
\end{Prop}
\begin{proof}
Set $\bv=\bvc$ and $\bv_\Delta=\bvc_\Delta$. There exists $\delta>0$ for which we have graph$(c+\bv_x)\supset pr\mathcal{U}^i_{loc}(\delta),pr\mathcal{U}^j_{loc}(\delta)$ and $\bv$ is smooth on $U^i(\delta)\cup U^j(\delta)$ including its boundary. Throughout this paper, $\epsilon_1(\Delta),\epsilon_2(\Delta),\ldots$ are numbers that tend to $0$ as $\Delta\to0$, and $b_1,b_2,\ldots$ are constants. Since $(\bv_\Delta)_x\to\bv_x$ a.e. as $\Delta\to0$,  we have  with small $\epsilon>0$,
\begin{eqnarray}\label{estimate}
\sup_{\R^2}|\bv_\Delta-\bv|<\epsilon_1(\Delta)\,\,\,\,\,\mbox{ and }\,\,\,\,\, \sup_{U^i(\delta+\epsilon)\cup U^j(\delta+\epsilon)}|(\bv_\Delta)_x-\bv_x|<\epsilon_1(\Delta).
\end{eqnarray}
\indent Now we investigate the solutions around $\gamma^\ast_i$. The same investigation is possible around $\gamma^\ast_j$. 
Take a $\Z^2$-periodic $C^2$-function $v:\R^2\to\R$ with the derivatives uniformly bounded in $\R^2$ such that 
$$\mbox{$v|_{U^i(\delta)}=\bv|_{U^i(\delta)}$ \,and \,$|H_p(x,t,c+v_x(x,t))|\le\lambda_1^{-1}$ in $\R^2$.}$$
Set 
$$g(x,t):=v_t+H(x,t,c+v_x(x,t))-h(c ),$$
where $g=0$ on $U^i(\delta)$. Then, $v^k_{m+1}:=v(x_{m+1},t_k)$ solves the difference equation
\begin{eqnarray}\label{eqnv} \quad\,\,\,\,\,
&&D_tv^{k+1}_m+H(x_m,t_k,c+D_xv^{k}_{m+1})=h(c )-\frac{\dx}{2\lambda}A(x_m,t_k)+g(x_m,t_k)+\varepsilon(x_m,t_k), \\\nonumber
&&A(x,t)=v_{xx}(x,t)-\lambda^2 v_{tt}(x,t),\,\,\,\varepsilon(x_m,t_k)=o(\dx),
\end{eqnarray}
and therefore it satisfies the equality  for any $n,l,l'<l$
\begin{eqnarray}\label{v_Delta}
v^{l+1}_n&=&\inf_{\xi}E_{\mu(\cdot;\xi)}\Big[\sum_{l'<k\le l+1}\{L^{(c )}(\gamma^k,t_{k-1},\xi^k_{m(\gamma^k)})-\frac{\dx}{2\lambda}A(\gamma^k,t_{k-1})\\\nonumber
&&\qquad\qquad+g(\gamma^k,t_{k-1})+\varepsilon(\gamma^k,t_{k-1}) \}\dt +v^{l'}_{m(\gamma^{l'})}\Big]+h(c )(t_{l+1}-t_{l'}).
\end{eqnarray}
Let $\bv^k_{m+1}$ denote the difference solution that yields $\bv_\Delta$. Define $w^k_{m+1}$ on $\mathcal{G}_{odd}$ as 
\[
w^k_{m+1}:= \left\{ \begin{array}{ll}
\bv^k_{m+1}\mbox{\,\,\,\, on $\tilde{U}^i(\delta):=U^i(\delta)\cap\mathcal{G}_{odd}$}\\
v^{k}_{m+1}+\bv^k_{m^\ast(k)+1}-v^k_{m^\ast(k)+1} \mbox{\,\,\,\, for ${m^\ast(k)}<{m}$},\\
v^{k}_{m+1}+\bv^k_{m_\ast(k)+1}-v^k_{m_\ast(k)+1} \mbox{\,\,\,\, for ${m}<{m_\ast(k)}$},
\end{array} \right.
\]
where $(x_{m^\ast(k)+1},t_k)$ and $(x_{m_\ast(k)+1},t_k)$ stand for the right end and left end of $\tilde{U}^i(\delta)|_{t=t_k}$, respectively. Note that difference between space variables of the nearest two end points of $\tilde{U}^i(\delta)$ is $\pm\dx$, due to $|\gamma^\ast_i{}'(s)|=|H_p(\gamma^\ast_i(s),s,c+\bv_x(\gamma^\ast_i(s),s))|<\lambda^{-1}$. It follows from (\ref{estimate}) that 
\begin{eqnarray}\label{w-v}
|w^k_{m+1}-v^k_{m+1}|<\epsilon_1(\Delta)\mbox{\quad on $\mathcal{G}_{odd}$.}
\end{eqnarray}
Set 
$$f(x_m,t_k):=D_tw^{k+1}_m+H(x_m,t_k,c+D_xw^k_{m+1})-h_\Delta(c ),$$
 where $f(x_m,t_k)=0$ for $m$ such that $(x_{m\pm1},t_k)\in\tilde{U}^i(\delta)$. Then, $w^k_{m+1}$ solves the difference equation 
\begin{eqnarray}\label{eqnw} \quad\,\,\,\,\,
D_tw^{k+1}_m+H(x_m,t_k,c+D_xw^{k}_{m+1})=h_\Delta(c )+f(x_m,t_k)
\end{eqnarray}
and therefore it satisfies for any $n,l,l'<l$
\begin{eqnarray}\label{w_Delta}\,\,\,
w^{l+1}_n&=&\inf_{\xi}E_{\mu(\cdot;\xi)}\Big[\sum_{l'<k\le l+1}\{L^{(c )}(\gamma^k,t_{k-1},\xi^k_{m(\gamma^k)})+f(\gamma^k,t_{k-1})\}\dt+w^{l'}_{m(\gamma^{l'})}\Big]\\\nonumber
&&+h_\Delta(c )(t_{l+1}-t_{l'}).
\end{eqnarray}
We will show 
\begin{eqnarray}\label{f-g}
|f(x_m,t_k)-g(x_m,t_k)|<\epsilon_2(\Delta)\mbox{\quad on $\mathcal{G}_{even}$}.
\end{eqnarray}
For $m_\ast(k)$, denoted by $m_\ast$, we have the following two cases: 

(i) $(x_{m_\ast},t_{k+1})\not\in \tilde{U}^i(\delta)$, which implies $(x_{m_\ast+2},t_{k+1})\in \tilde{U}^i(\delta)$. Then, for $m\le m_\ast$, we have 
\begin{eqnarray*}
D_xw^k_{m+1}&=&\frac{1}{2\dx}\{(v^k_{m+1}+\bv^k_{m_\ast+1}-v^k_{m_\ast+1})-(v^k_{m-1}+\bv^k_{m_\ast+1}-v^k_{m_\ast+1})\}\\
&=&D_xv^k_{m+1},\\
D_tw^{k+1}_{m}&=&\frac{1}{\dt}\Big\{  v^{k+1}_{m}+ \bv^{k+1}_{m_\ast+2}-v^{k+1}_{m_\ast+2}\\
&&-\frac{1}{2}(v^k_{m-1}+\bv^k_{m_\ast+1}-v^k_{m_\ast+1}+v^k_{m+1}+\bv^k_{m_\ast+1}-v^k_{m_\ast+1} )  \Big\}\\
&=&D_tv^{k+1}_{m}+D_t\bv^{k+1}_{m_\ast+2}-D_tv^{k+1}_{m_\ast+2}+\lambda^{-1}(D_x\bv^k_{m_\ast+3}-D_xv^k_{m_\ast+3}).
\end{eqnarray*}
Hence, with (\ref{eqnv}) and (\ref{eqnw}), we obtain
\begin{eqnarray}\label{fgfg1}
&&f(x_m,t_k)-g(x_{m},t_k)=H(x_{m_\ast+2},t_k,c+D_xv^k_{m_\ast+3})-H(x_{m_\ast+2},t_k,c+D_x\bv^k_{m_\ast+3})\\\nonumber
&&\qquad \qquad \qquad \qquad \qquad +\lambda^{-1}(D_x\bv^k_{m_\ast+3}-D_xv^k_{m_\ast+3})-\frac{\dx}{2\lambda}A(x_{m},t_k)+\varepsilon(x_{m},t_k)\\\nonumber
&&\qquad \qquad \qquad \qquad \qquad +\frac{\dx}{2\lambda}A(x_{m_\ast+2},t_k)-\varepsilon(x_{m_\ast+2},t_k).
\end{eqnarray}
\indent(ii) $(x_{m_\ast},t_{k+1})\in \tilde{U}^i(\delta)$. Then, for $m\le m_\ast$, we have 
\begin{eqnarray*}
D_xw^k_{m+1}&=&\frac{1}{2\dx}\{(v^k_{m+1}+\bv^k_{m_\ast+1}-v^k_{m_\ast+1})-(v^k_{m-1}+\bv^k_{m_\ast+1}-v^k_{m_\ast+1})\}\\
&=&D_xv^k_{m+1},\\
D_tw^{k+1}_{m}&=&\frac{1}{\dt}\Big\{  v^{k+1}_{m}+ \bv^{k+1}_{m_\ast}-v^{k+1}_{m_\ast}\\
&&-\frac{1}{2}(v^k_{m-1}+\bv^k_{m_\ast+1}-v^k_{m_\ast+1}+v^k_{m+1}+\bv^k_{m_\ast+1}-v^k_{m_\ast+1} )  \Big\}\\
&=&D_tv^{k+1}_{m}+D_t\bv^{k+1}_{m_\ast}-D_tv^{k+1}_{m_\ast}+\lambda^{-1}(D_x\bv^k_{m_\ast+1}-D_xv^k_{m_\ast+1}).
\end{eqnarray*}
Hence, with (\ref{eqnv}) and (\ref{eqnw}), we obtain
\begin{eqnarray}\label{fgfg2}
&&f(x_m,t_k)-g(x_{m},t_k)=H(x_{m_\ast},t_k,c+D_xv^k_{m_\ast+1})-H(x_{m_\ast},t_k,c+D_x\bv^k_{m_\ast+1})\\\nonumber
&&\qquad \qquad \qquad \qquad \qquad +\lambda^{-1}(D_x\bv^k_{m_\ast+1}-D_xv^k_{m_\ast+1})+\frac{\dx}{2\lambda}A(x_{m},t_k)-\varepsilon(x_{m},t_k)\\\nonumber
&&\qquad \qquad \qquad \qquad \qquad -\frac{\dx}{2\lambda}A(x_{m_\ast},t_k)+\varepsilon(x_{m_\ast},t_k)+g(x_{m_\ast},t_k),
\end{eqnarray}
where $g(x_{m_\ast},t_k)\to g(x_{m_\ast+2},t_k)=0$ as $\Delta\to0$. 

 Similar calculation is possible with $m^\ast(k)$.  Therefore, noting that $D_xv^k_{m+1}=v_x(x_m,t_k)+o(\dx)$ and $D_x\bv^k_{m+1}=(\bv_\Delta)_x(x_m,t_k)$, we obtain (\ref{f-g}) through (\ref{estimate}).  

We state the law of large numbers for the minimizing random walks for (\ref{v_Delta}) and (\ref{w_Delta}), where they tend to genuine minimizing curves for exact viscosity solutions. 
\begin{Lemma}\label{LLN1}
(1) For each $x\in X:=\{x\in\R\,|\,|x-\gamma^\ast_i(0)|\le\tilde{\delta}<\delta\}$, $t<0$ and $\varepsilon>0$, define
\begin{eqnarray*} 
&&\mathcal{L}(r ):=\int^0_{t}\{L^{(c )}(r(s),s,r'(s))+g(r(s),s)\}ds+v(r(t),t)+h(c )(-t),\\
&&\gamma^\ast_{x}\mbox{: minimizer for} \inf_{\gamma\in AC,\gamma(0)=x}\mathcal{L}(\gamma),\\
&&D^\varepsilon_{x}:=\{r:[t,0]\to\R\,|\,|r'(s)|<\lambda^{-1}, r(0)=x,\mathcal{L}(r )-\mathcal{L}(\gamma^\ast_{x})\le\varepsilon \}.
\end{eqnarray*}
Then, we have 
$$ \sup_{x\in X}\sup_{r\in D^\varepsilon_{x}}\norm r- \gamma^\ast_{x}\norm_{C^0([t,0])}\to0\mbox{ \,\,\, as $\varepsilon\to0$.}$$

(2) Let $\eta_\Delta(\gamma)$ denote the linear interpolation of (\ref{etaeta}). For each $x_{n+1}\in X$ such that $(x_{n+1},0)\in\mathcal{G}_{odd}$, $\nu>0$ and the minimizing random walk $\gamma=\gamma(x_{n+1},0;\xi^\ast)$ within $[t_{l'},0]$ for $v^0_{n+1}$ or $w^0_{n+1}$,  define
\begin{eqnarray*}
\Omega^\nu_\Delta(x_{n+1}):=\{\gamma\in\Omega\,\,|\,\, \norm\eta_\Delta(\gamma)-\gamma^\ast_{x_{n+1}}\norm_{C^0([t_{l'},0])}\le\nu \}.
\end{eqnarray*}
Then, we have for each fixed $\nu>0$,  
$$\inf_{x_{n+1}}prob(\Omega^\nu_\Delta(x_{n+1}))\to1\mbox{ as $\Delta\to0$}.$$
In particular, the linearly interpolated averaged path $\bar{\gamma}_\Delta$ of the minimizing random walk ($\bar{\gamma}_\Delta$ is also the average of $\eta_\Delta(\gamma)$. See Theorem 3.2, \cite{Soga1}) satisfies $\norm \bar{\gamma}_\Delta-\gamma^\ast_{x_{n+1}}\norm_{C^0([t_{l'},0])}\to0$ as $\Delta\to0$, uniformly with respect to $x_{n+1}$.   
\end{Lemma}
\medskip 

\noindent{\bf Remark.} {\it Since $v(x,0)=\inf_{\gamma\in AC,\gamma(0)=x}\mathcal{L}(\gamma)$ and $v=\bv$, $g=0$ on $U^i(\delta)$ with differentiability of $\bv$, the minimizer $\gamma^\ast_x$ for $\inf_{\gamma\in AC,\gamma(0)=x}\mathcal{L}(\gamma)$ coincides with the unique minimizer $\gamma^\ast$ for $\bv(x,0)$ within a small time interval $[\epsilon,0]$. Since  the dynamics on the unstable manifold implies that $\gamma^\ast(s)$ tends to $\gamma^\ast_i(s)$ as $s\to-\infty$, the points $(\gamma^\ast_x(s),s)$ never come out of $U^i(\delta)$. Hence $\gamma^\ast_x$ is equal to $\gamma^\ast$ on $[t,0]$, and is the unique minimizer for $ \inf_{\gamma\in AC,\gamma(0)=x}\mathcal{L}(\gamma)$. It follows from hyperbolicity of $\gamma^\ast_i$ that there exists $b_1>0$ for which we have for $s\le0$,
\begin{eqnarray}\label{hyperbolic}
|\gamma^\ast_x(s)-\gamma^\ast_i(s)|\le |\gamma^\ast_x(0)-\gamma^\ast_i(0)|e^{b_1s}.
\end{eqnarray}}
\begin{proof}
(1) Proceeding by a proof via contradiction, we assume that there exists a sequence $\varepsilon_j\to0$ as $j\to\infty$ for which we have $b_2>0$ such that 
$$\sup_{x\in X}\sup_{r\in D^{\varepsilon_j}_{x}}\norm r- \gamma^\ast_{x}\norm_{C^0}\ge 3b_2\mbox{\,\, for all $j$.}$$
 We can take $x_j\in X$ and $r_j\in D^{\varepsilon_j}_{x_j}$ such that 
 $$\norm r_j- \gamma^\ast_{x_j}\norm_{C^0}\ge 2b_2\mbox{\,\, for all $j$.}$$ 
We have a convergent subsequence of $\{x_j\}$, still denoted by $\{x_j\}$, that tends to $x_\sharp\in X$ as $j\to\infty$. Since $\gamma^\ast_{x_j}{}'(0)=H_p(x_j,0,c+\bv_x(x_j,0))\to\gamma^\ast_{x_\sharp}{}'(0)=H_p(x_\sharp,0,c+\bv_x(x_\sharp,0))$ as $j\to\infty$, we have $\gamma^\ast_{x_j}\to\gamma^\ast_{x_\sharp}$, $\gamma^\ast_{x_j}{}'\to\gamma^\ast_{x_\sharp}{}'$ uniformly and $\mathcal{L}(\gamma^\ast_{x_j})\to\mathcal{L}(\gamma^\ast_{x_\sharp})$, as $j\to\infty$.   Hence, we see that 
\begin{eqnarray*}
|\mathcal{L}(r_j)-\mathcal{L}(\gamma^\ast_{x_\sharp})|&\le& |\mathcal{L}(r_j)-\mathcal{L}(\gamma^\ast_{x_j})|+|\mathcal{L}(\gamma^\ast_{x_j})-\mathcal{L}(\gamma^\ast_{x_\sharp})|\\
&\le&\varepsilon_j+|\mathcal{L}(\gamma^\ast_{x_j})-\mathcal{L}(\gamma^\ast_{x_\sharp})|\to0\mbox{\,\, as $j\to\infty$.}
\end{eqnarray*}
 Therefore, we have a subsequence of $\{r_j\}$, still denoted by $\{r_j\}$, that tends to a curve $r_\sharp$ uniformly as $j\to\infty$ with $\mathcal{L}(r_\sharp)=\mathcal{L}(\gamma^\ast_{x_\sharp})$. Since 
$$\norm r_\sharp-\gamma^\ast_{x_\sharp}\norm_{C^0}\ge -\norm r_\sharp-r_j\norm_{C^0}+\norm r_j-\gamma^\ast_{x_j}\norm_{C^0}-\norm \gamma^\ast_{x_j}-\gamma^\ast_{x_\sharp}\norm_{C^0}\ge b_2,$$
 there are two minimizing curves, which is a contradiction.  

(2) For the minimizing random walk $\gamma=\gamma(x_{n+1},0;\xi^\ast)$ for $v^0_{n+1}$, we have with (1) of Proposition \ref{eta}, 
\begin{eqnarray*} 
v^0_{n+1}&=&E_{\mu(\cdot;\xi^\ast)}\Big[ \sum_{l'<k\le 0} \{L^{(c )}(\gamma^k,t_{k-1},\xi^k_{m(\gamma^k)})+g(\gamma^k,t_{k-1}) \\
&&-\frac{\dx}{2\lambda}A(\gamma^k,t_{k-1})+\varepsilon(\gamma^k,t_{k-1}) \}\dt+v(\gamma^{l'},t_{l'}) \Big] + h(c )(-t_{l'})\\
&=&E_{\mu(\cdot;\xi^\ast)}\Big[ \sum_{l'<k\le 0} \{L^{(c )}(\eta^k(\gamma),t_{k-1},\xi^k_{m(\gamma^k)})+g(\eta^k(\gamma),t_{k-1})  \}\dt\\
&&+v(\eta^{l'}(\gamma),t_{l'}) \Big] + h(c )(-t_{l'})+\epsilon_4(\Delta)\\
&=&E_{\mu(\cdot;\xi^\ast)}\Big[ \int^0_{t_{l'}}\{L^{(c )}(\eta_\Delta(\gamma)(s),s,\eta_\Delta(\gamma)'(s))+g(\eta_\Delta(\gamma)(s),s)\}ds\\
&&+v(\eta_\Delta(\gamma)(t_{l'}),t_{l'}) \Big] + h(c )(-t_{l'})+\epsilon_5(\Delta)\\
&=&E_{\mu(\cdot;\xi^\ast)}\Big[ \mathcal{L}(\eta_\Delta(\gamma))\Big]+\epsilon_5(\Delta).
\end{eqnarray*}
On the other hand, $v^0_{n+1}=v(x_{n+1},0)=\mathcal{L}(\gamma^\ast_{x_{n+1}})$. Hence, we have    
\begin{eqnarray}\label{e66}
0\le E_{\mu(\cdot;\xi^\ast)}\Big[ \mathcal{L}(\eta_\Delta(\gamma))-\mathcal{L}(\gamma^\ast_{x_{n+1}})\Big]\le\epsilon_6(\Delta).
\end{eqnarray}
Here, $\epsilon_6(\Delta)>0$ are independent of $x_{n+1}$.  Set
$$\Omega^+:=\{\gamma\in\Omega\,|\,  \mathcal{L}(\eta_\Delta(\gamma))-\mathcal{L}(\gamma^\ast_{x_{n+1}})\ge \epsilon_6(\Delta)^{\frac{1}{2}} \}.$$
Then, by (\ref{e66}), we have $prob(\Omega^+)\le\epsilon_6(\Delta)^{1/2}$. 
It follows from (1) of this lemma that  there exists $\varepsilon_0(\nu)>0$ for which, if $|\Delta|$ is small enough to realize  $\epsilon_6(\Delta)^{1/2}\le \varepsilon_0(\nu)$, we have $\norm \eta_\Delta(\gamma)-\gamma^\ast_{x_{n+1}}\norm_{C^0}\le\nu$ for all $\gamma\in\Omega\setminus\Omega^+$. 
Therefore, we obtain  $\Omega\setminus\Omega^+\subset \Omega^\nu_\Delta(x_{n+1})$ and $1-\epsilon_6(\Delta)^{1/2}\le prob(\Omega_\Delta^\nu(x_{n+1}))$. 

Similarly, for the minimizing random walk $\gamma=\gamma(x_{n+1},0;\xi^\ast)$ for $w^0_{n+1}$, we have with(\ref{w-v}), (\ref{f-g}) and (3) of Proposition  \ref{preliminary}, 
\begin{eqnarray*} 
w^0_{n+1}&=&E_{\mu(\cdot;\xi^\ast)}\Big[ \sum_{l'<k\le 0} \{L^{(c )}(\gamma^k,t_{k-1},\xi^k_{m(\gamma^k)})+f(\gamma^k,t_{k-1}) \}\dt+w^{l'}_{m(\gamma^{l'})} \Big] \\&&
+ h_\Delta(c )(-t_{l'})\\
&=&E_{\mu(\cdot;\xi^\ast)}\Big[ \sum_{l'<k\le 0} \{L^{(c )}(\gamma^k,t_{k-1},
\xi^k_{m(\gamma^k)})+g(\gamma^k,t_{k-1}) \}\dt
+v(\gamma^{l'},t_{l'}) \Big] \\
&&+ h(c )(-t_{l'})+\epsilon_7(\Delta)\\
&=&E_{\mu(\cdot;\xi^\ast)}\Big[ \mathcal{L}(\eta_\Delta(\gamma))\Big]+\epsilon_8(\Delta),\\
|w^0_{n+1}-v^0_{n+1}|&=&|w^0_{n+1}-\mathcal{L}(\gamma^\ast_{x_{n+1}})|\le \epsilon_1(\Delta).
\end{eqnarray*}
Therefore we have an estimate similar to (\ref{e66}), and we may follow the same way as the above. 
\end{proof}
\indent Now, we compare (\ref{v_Delta}) and (\ref{w_Delta}). Take $n$ so that $w^0_{n+1}-v^0_{n+1}$ is equal to 
$$\max_{\tilde{U}^i(\delta/3)|_{k=0}}(w^0_{m+1}-v^0_{m+1}).$$
Let $T\in\N$ be such that $qT>\frac{1}{b_1}\log4$, let $l':=-qT\cdot 2K$ and let $\gamma^\ast_{x_{n+1}}:[-qT,0]\to \R$ be the minimizing curve for $v(x_{n+1},0)$.  Then, (\ref{hyperbolic}) implies 
\begin{eqnarray}\label{decay}
|\gamma^\ast_{x_{n+1}}(-qT)-\gamma^\ast_i(-qT)|\le \frac{1}{4}|\gamma^\ast_{x_{n+1}}(0)-\gamma^\ast_i(0)|.
\end{eqnarray}
Let $\xi^\ast$ be the minimizer for $v^0_{n+1}$. Then, it follows from (\ref{v_Delta}) and (\ref{w_Delta}) that 
\begin{eqnarray}\nonumber
w^0_{n+1}-v^0_{n+1}&\le& E_{\mu(\cdot;\xi^\ast)}\Big[ \sum_{l'<k\le 0}\big\{ f(\gamma^k,t_{k-1})-g(\gamma^k,t_{k-1})+\frac{\dx}{2\lambda}A(\gamma^k,t_{k-1})\\\nonumber
&&-\varepsilon(\gamma^k,t_{k-1})   \big\}\dt + (w^{l'}_{m(\gamma^{l'})}  - v^{l'}_{m(\gamma^{l'})}) \Big]+(h_\Delta(c )-h(c ))qT\\\nonumber
&=&\frac{\dx}{2\lambda}\int^0_{-qT}A(\gamma^\ast_{x_{n+1}}(s),s)ds\\\label{111}
&&+\frac{\dx}{2\lambda}E_{\mu(\cdot;\xi^\ast)}\Big[\int^0_{-qT}A(\eta_\Delta(\gamma)(s),s)ds\Big]-\frac{\dx}{2\lambda}\int^0_{-qT}A(\gamma^\ast_{x_{n+1}}(s),s)ds\\\label{112}
&&+\frac{\dx}{2\lambda}E_{\mu(\cdot;\xi^\ast)}\Big[ \sum_{l'<k\le 0} A(\gamma^k,t_{k-1})\dt-\int^0_{-qT}A(\eta_\Delta(\gamma)(s),s)ds\Big]\\\label{113}
&&-E_{\mu(\cdot;\xi^\ast)}\Big[ \sum_{l'<k\le 0} \varepsilon(\gamma^k,t_{k-1})\dt\Big]\\\label{114}
&& +E_{\mu(\cdot;\xi^\ast)}\Big[ \sum_{l'<k\le 0}\{f(\gamma^k,t_{k-1})-g(\gamma^k,t_{k-1})\}\dt\Big]  \\\label{115}
&&+E_{\mu(\cdot;\xi^\ast)}\Big[ (w^{l'}_{m(\gamma^{l'})}  - v^{l'}_{m(\gamma^{l'})})-(w^0_{n+1}-v^0_{n+1}) \Big] \\\nonumber
&&+(w^0_{n+1}-v^0_{n+1})+(h_\Delta(c )-h(c ))qT.
\end{eqnarray}
The terms in (\ref{111})--(\ref{115}) are denoted by $R_1,\ldots,R_5$, respectively. Note that $A(x,t)$ is uniformly continuous. By (2) of Lemma \ref{LLN1} and continuity of $A$, we obtain $R_1\le \epsilon_{11}(\Delta)\dx$. 
It follows from (1) of Proposition \ref{eta} that 
$$prob\{\gamma\in\Omega\,|\,|\eta^k(\gamma)-\gamma^k|>\dx^{\frac{1}{4}}\}\le b_3\dx^\frac{1}{4},$$
and hence we obtain $R_2\le\epsilon_9(\Delta)\dx$. Since $\varepsilon=o(\dx)$, we have $R_3\le \epsilon_{10}(\Delta)\dx$. 
Since the minimizer $\xi^\ast$ is given as $\xi^\ast{}^{k+1}_{m}=H_p(x_m,t_{k},c+D_xv^k_{m+1})$ and satisfies the Lipschitz condition in (2) of 
Proposition~\ref{eta}, Chebychev's inequality yields for each $\alpha>0$,
\begin{eqnarray}\label{Che}
prob\{\gamma\in\Omega\,\,|\,\,|\gamma^k-\bar{\gamma}^k|>\alpha  \}\le \frac{\sigma^k}{\alpha^2}\le b_4\frac{\dx}{\alpha^2} \mbox{\,\,\,\, for all $l'\le k\le 0$}.
\end{eqnarray}
Since the averaged path $\bar{\gamma}_\Delta$ of the minimizing random walk is $C^0$-close to $\gamma^\ast_{x_{n+1}}$, and $\gamma^\ast_{x_{n+1}}$ tends to $\gamma^\ast_i$, the points $(\bar{\gamma}_\Delta(s),s)$ always stay in $U^i(4\delta/9)$ for small $\Delta$. Hence, we have $\{(x,s)\in\R^2\,|\, |x-\bar{\gamma}_\Delta(s)|\le\delta/2 \}\subset U^i(\delta)$. With the fact $f(\gamma^k,t_{k-1})-g(\gamma^k,t_{k-1})=0$ on $\tilde{U}^i(\delta)$, (\ref{f-g}) and (\ref{Che}) with $\alpha=\delta/2$,  we obtain
\begin{eqnarray*} 
R_4\le \epsilon_{12}(\Delta)\frac{\dx}{(\delta/2)^2}.
\end{eqnarray*}
Due to the choice of $n$, we have $ (w^{l'}_{m(\gamma^{l'})}  - v^{l'}_{m(\gamma^{l'})})-(w^0_{n+1}-v^0_{n+1})<0$ for $\gamma$ such that $\gamma^{l'}$ is contained in the $\delta/3$-neighborhood of $\gamma^\ast_i(-qT)$. We see that the $\delta/6$-neighborhood of $\bar{\gamma}_\Delta(-qT)$ is contained in the $\delta/3$-neighborhood of $\gamma^\ast_i(-qT)$. In fact, it follows from (\ref{decay}) that we have $|\gamma^\ast_{x_{n+1}}(-qT)-\gamma^\ast_i(-qT)|\le \delta/12$, and hence we have $|\bar{\gamma}_\Delta(-qT)-\gamma^\ast_i(-qT)|\le 1.5\delta/12$ for small $\Delta$. By (\ref{w-v}) and (\ref{Che}) with $\alpha=\delta/6$, we obtain 
\begin{eqnarray*}
R_5\le \epsilon_{13}(\Delta)\frac{\dx}{(\delta/6)^2}. 
\end{eqnarray*}
It follows from (\ref{hyperbolic}) and the relation $v_{xx}=\zeta^i_{xx}, v_{tt}=\zeta^i_{tt}$ on $U^i(\delta)$ that  
$$\left|\int^0_{-qT}A(\gamma^\ast_{x_{n+1}}(s),s)ds- T\Gamma^i\right|\le \nu(\delta),$$
where $\nu(\delta)$ comes from the modulus of the continuity of $A$. Therefore we obtain
\begin{eqnarray}\label{upper1}
(h(c )-h_\Delta(c ))qT\le T\Gamma^i\frac{\dx}{2\lambda}+\nu(\delta)\frac{\dx}{2\lambda}+\frac{\epsilon_{14}(\Delta)}{\delta^2}\dx+\epsilon_{15}(\Delta)\dx.
\end{eqnarray}
We remark that if we do not introduce $w^k_{m+1}$, or if we do not choose the above $n$, it is not easy to see that $\Gamma^i$ is the principle term in (\ref{upper1}).   

\indent Next, we will obtain a converse estimate with the minimizer $\xi^\ast$ for $w^{0}_{n+1}$, where we take $n$ so that $w^0_{n+1}-v^0_{n+1}$ is equal to 
$$\min_{\tilde{U}^i(\delta/3)|_{k=0}}(w^0_{m+1}-v^0_{m+1}).$$
It follows from (\ref{v_Delta}) and (\ref{w_Delta}) that
\begin{eqnarray*}\nonumber
w^0_{n+1}-v^0_{n+1}&\ge& E_{\mu(\cdot;\xi^\ast)}\Big[ \sum_{l'<k\le 0}\big\{ f(\gamma^k,t_{k-1})-g(\gamma^k,t_{k-1})+\frac{\dx}{2\lambda}A(\gamma^k,t_{k-1})\\\nonumber
&&-\varepsilon(\gamma^k,t_{k-1})   \big\}\dt + (w^{l'}_{m(\gamma^{l'})}  - v^{l'}_{m(\gamma^{l'})}) \Big]+(h_\Delta(c )-h(c ))qT\\\nonumber
&=&\frac{\dx}{2\lambda}\int^0_{-qT}A(\gamma^\ast_{x_{n+1}}(s),s)ds\\
&&+\frac{\dx}{2\lambda}E_{\mu(\cdot;\xi^\ast)}\Big[ \int^0_{-qT}A(\eta_\Delta(\gamma)(s),s)ds\Big]-\frac{\dx}{2\lambda}\int^0_{-qT}A(\gamma^\ast_{x_{n+1}}(s),s)ds\\\
&&+\frac{\dx}{2\lambda}E_{\mu(\cdot;\xi^\ast)}\Big[ \sum_{l'<k\le 0} A(\gamma^k,t_{k-1})\dt-\int^0_{-qT}A(\eta_\Delta(\gamma)(s),s)ds\Big]\\
&&-E_{\mu(\cdot;\xi^\ast)}\Big[ \sum_{l'<k\le 0} \varepsilon(\gamma^k,t_{k-1})\dt\Big]\\
&&+E_{\mu(\cdot;\xi^\ast)}\Big[ \sum_{l'<k\le 0}\{f(\gamma^k,t_{k-1})-g(\gamma^k,t_{k-1})\}\dt\Big]\\
&&+E_{\mu(\cdot;\xi^\ast)}\Big[ (w^{l'}_{m(\gamma^{l'})}  - v^{l'}_{m(\gamma^{l'})})-(w^0_{n+1}-v^0_{n+1}) \Big] \\\nonumber
&&+(w^0_{n+1}-v^0_{n+1})+(h_\Delta(c )-h(c ))qT.
\end{eqnarray*}
Since the averaged path $\bar{\gamma}_\Delta$ of the minimizing random walk is $C^0$-close to $\gamma^\ast_{x_{n+1}}$, and $\gamma^\ast_{x_{n+1}}$ tends to $\gamma^\ast_i$, the points $(\bar{\gamma}_\Delta(s),s)$ always stay in $U^i(4\delta/9)$ for small $\Delta$. Hence, by (A4),  the minimizer $\xi^\ast$, which is given by $\xi^\ast{}^{k+1}_{m}=H_p(x_m,t_{k},c+D_xw^k_{m+1})$, satisfies the Lipschitz condition in (2) of Proposition \ref{eta}. Therefore, with the estimate $\sigma^k=O(\Delta x)$, we can repeat the same argument as the above for the lower bound of each term.  In particular, due to the choice of the $n$, we have $ (w^{l'}_{m(\gamma^{l'})}  - v^{l'}_{m(\gamma^{l'})})-(w^0_{n+1}-v^0_{n+1})>0$ for $\gamma$ such that $\gamma^{l'}$ is contained in the $\delta/3$-neighborhood of $\gamma^\ast_i(-qT)$, and the $\delta/6$-neighborhood of $\bar{\gamma}_\Delta(-qT)$ is contained in the $\delta/3$-neighborhood of $\gamma^\ast_i(-qT)$. Thus we obtain
\begin{eqnarray}\label{lower1}
(h(c )-h_\Delta(c ))qT\ge T\Gamma^i\frac{\dx}{2\lambda}-\nu(\delta)\frac{\dx}{2\lambda}-\frac{\epsilon_{14}(\Delta)}{\delta^2}\dx-\epsilon_{15}(\Delta)\dx.
\end{eqnarray}
\indent In this way, we also obtain for $\gamma^\ast_j$,
\begin{eqnarray}\label{upper2}
(h(c )-h_\Delta(c ))qT\le T\Gamma^j\frac{\dx}{2\lambda}+\nu(\delta)\frac{\dx}{2\lambda}+\frac{\epsilon_{14}(\Delta)}{\delta^2}\dx+\epsilon_{15}(\Delta)\dx,\\\label{lower2}
(h(c )-h_\Delta(c ))qT\ge T\Gamma^j\frac{\dx}{2\lambda}-\nu(\delta)\frac{\dx}{2\lambda}-\frac{\epsilon_{14}(\Delta)}{\delta^2}\dx-\epsilon_{15}(\Delta)\dx.
\end{eqnarray}
If $\Gamma^i\neq\Gamma^j$ and if we take small $\delta>0$, there is a contradiction in (\ref{upper1}),(\ref{lower1}),(\ref{upper2}) and (\ref{lower2}) as $\Delta\to0$.   This ends the proof of Proposition \ref{key}.
\end{proof}
 Currently, we do not have any appropriate estimate of $\sigma^k$ without (A4) for the rate of the law of large numbers in the minimizing random walk for $w^0_{n+1}$. A possible way to remove (A4) would be to sharpen Chebychev's inequality. However, this seems not easy, because our random walk is not the sum of i.i.d. random variables.   

Next we show that, for certain values of $c$, the limit of any convergent subsequence $\{\bv^{(c )}_\Delta\}$ has transition points at $\gamma^\ast_{i^\ast}$. For this purpose, we introduce the following symbols: 
Let $x^+_k$ (resp. $x^-_k$) $\in\{ pr\gamma^\ast_i(0),\ldots,pr\gamma^\ast_i(q-1)\}_{i\neq i^\ast}$ be the nearest point to $pr\gamma^\ast_{i^\ast}(k)$ on the right (resp. left) for $k=0,\ldots,q-1$ (if $\gamma^\ast_{i^\ast}(0)=0$, $x_0^-$ is the point of $\{ pr\gamma^\ast_i(0),\ldots,pr\gamma^\ast_i(q-1)\}_{i\neq i^\ast}$ that is nearest to $1$). Let $\mathcal{U}^\pm_k$ (reps. $\mathcal{S}^\pm_k$) be the segment between $pr\gamma^\ast_{i^\ast}(k)$ and $x^\pm_k$ of  the projected unstable (resp. stable) manifold at $t=0$ of $\gamma^\ast_{i^\ast}$. Set 
$$d^\pm:=\sum_{0\le k<q}[\mbox{area surrounded by $\mathcal{U}^\pm_k$ and $\mathcal{S}^\pm_k$}].$$
It follows from Subsection 3.3 that, for each $c\in J^-:=[c_0,c_1-d^-]$, there exists $\bvc$  such that graph$(c+\bvc_x)|_{t=0}\supset\cup_{0\le k<q}\mathcal{U}^-_k$. Similarly, for each $c\in J^+:=[c_0+d^+,c_1]$, there exists $\bvc$  such that graph$(c+\bvc_x)|_{t=0}\supset\cup_{0\le k<q}\mathcal{U}^+_k$. Since there exists $\bvc$  such that graph$(c+\bvc_x)|_{t=0}\supset\cup_{0\le k<q}(\mathcal{U}^-_k\cup\mathcal{U}^+_k)$, we see that $J^-\cap J^+\neq\phi$. 
\begin{Prop}\label{existence} 
For each $c\in J^-\cap J^+$, the limit of any convergent subsequence $\{\bv^{(c )}_\Delta\}$ has transition points at $\gamma^\ast_{i^\ast}$. 
\end{Prop}
\begin{proof}
Let $c\in J^-\cap J^+$. Suppose that there exists a convergent subsequence $\{\bv_\Delta=\bv^{(c )}_\Delta \}$ whose limit $\bv$ does not have transition points at $\gamma^\ast_{i^\ast}$. The following is our strategy: Since $\bv$ has transition points at some $\gamma^\ast_i$ ($i\neq i^\ast$), we have the lower estimate (\ref{lower1}). Since $\bv$ does not have transition points at $\gamma^\ast_{i^\ast}$, the upper estimate (\ref{upper1}) with $i=i^\ast$ is not directly available.  However, due to the choice of $c$ and uniqueness of the value $h(c)$, we have the $\Z^2$-periodic viscosity solution $\hat{v}=\hat{v}^{(c )}$ of (\ref{HJ}) such that $\hat{v}(\gamma^\ast_{i^\ast}(0),0)=0$ and 
$$\mbox{graph}(c+\hat{v}_x)|_{t=0}\supset\bigcup_{0\le k<q}(\mathcal{U}^-_k\cup\mathcal{U}^+_k).$$  
Then, we can compare $\bv_\Delta$ and $\hat{v}$ with modification of the argument in Proposition \ref{key}, obtaining the upper estimate (\ref{upper1}) with $i=i^\ast$. (A3) yields contradiction as $\Delta\to0$.

Adding constants if necessary, we have $\bv_\Delta(\gamma^\ast_{i^\ast}(0),0)=\bv(\gamma^\ast_{i^\ast}(0),0)=0$. The entropy condition of $\bv_x$  implies the following three cases: graph$(c+\bv_x)|_{t=0}$ contains 

(i) $\mathcal{S}^-_k$ and a part of $\mathcal{U}^+_k$\,\,\, for $k=0,\ldots,q-1$,

(ii) a part of $\mathcal{U}^-_k$ and $\mathcal{S}^+_k$\,\,\, for $k=0,\ldots,q-1$,

(iii)  $\mathcal{S}^-_k$ and $\mathcal{S}^+_k$\,\,\, for $k=0,\ldots,q-1$.   

\noindent Take small $\delta>0$ so that 

(i) $\Rightarrow$ graph$(c+\bv_x)\supset pr\mathcal{U}^{i^\ast}_{loc}(\delta)|_{\{(x,t)\in\R^2 \,|\, \gamma^\ast_{i^\ast}(t)\le x\le \gamma^\ast_{i^\ast}(t)+\delta \}}$, 

(ii) $\Rightarrow$ graph$(c+\bv_x)\supset pr\mathcal{U}^{i^\ast}_{loc}(\delta)|_{\{(x,t)\in\R^2 \,|\, \gamma^\ast_{i^\ast}(t)-\delta\le x\le \gamma^\ast_{i^\ast}(t) \}}$, 

(iii) $\Rightarrow$ $\delta\le \min\{  \min_{t\in\R}|\gamma^\ast_{i^\ast-1}(t)- \gamma^\ast_{i^\ast}(t)|,\min_{t\in\R}|\gamma^\ast_{i^\ast}(t)-\gamma^\ast_{i^\ast+1}(t)|\}$.   

\noindent Following the proof of Proposition \ref{key}, we define a $\Z^2$-periodic $C^2$-function $v$ with $\hat{v}$, instead of $\bv$, so that 

(i) $\Rightarrow$ $v=\hat{v}$ on $V(\delta):=\{(x,t)\in\R^2\,\,|\,\,\gamma^\ast_{i^\ast-1}(t)\le x\le \gamma^\ast_{i^\ast}(t)+\delta \}$,
 
(ii) $\Rightarrow$ $v=\hat{v}$ on $V(\delta):=\{(x,t)\in\R^2\,\,|\,\,\gamma^\ast_{i^\ast}(t)-\delta\le x\le \gamma^\ast_{i^\ast+1}(t) \}$,

(iii) $\Rightarrow$ $v=\hat{v}$ on $V(\delta):=\{(x,t)\in\R^2\,\,|\,\,\gamma^\ast_{i^\ast-1}(t)\le x\le \gamma^\ast_{i^\ast+1}(t) \}$  (independent of $\delta$).

\noindent Let $\bv^k_{m+1}$ denote the difference solution that yields $\bv_\Delta$. We also define $w^k_{m+1}$ from $\bv^k_{m+1}$ and $v$ with $\tilde{V}(\delta):=V(\delta)\cap \mathcal{G}_{odd}$ as
\[
w^k_{m+1}:= \left\{ \begin{array}{ll}
\bv^k_{m+1}\mbox{\,\,\,\, on $\tilde{V}(\delta)$}\\
v^{k}_{m+1}+\bv^k_{m^\ast(k)+1}-v^k_{m^\ast(k)+1} \mbox{\,\,\,\, for ${m^\ast(k)}<{m}$},\\
v^{k}_{m+1}+\bv^k_{m_\ast(k)+1}-v^k_{m_\ast(k)+1} \mbox{\,\,\,\, for ${m}<{m_\ast(k)}$},
\end{array} \right.
\]
where $(x_{m^\ast(k)+1},t_k)$ and $(x_{m_\ast(k)+1},t_k)$ stand for the right end and left end of $\tilde{V}(\delta)|_{t=t_k}$, respectively.  Then we may have $g$ and $f$, as well as the difference equations for $v^k_{m+1}:=v(x_{m+1},t_k)$ and $w^k_{m+1}$, in the same way as that of the proof of Proposition \ref{key}.

 Since $\bv_x=\hat{v}_x$ on the boundary $\partial V(\delta)$, 
we have for the $\epsilon_{16}(\Delta)$-neighborhood of $\partial V(\delta)$, denoted by $B(\Delta)$,
$$\sup_{B(\Delta)}|\bv_x-\hat{v}_x|\le\epsilon_{17}(\Delta).$$ 
Since $(\bv_\Delta)_x\to\bv_x$ uniformly on $V(\delta)$ including its boundary as $\Delta\to0$, it follows from the continuity of $v_x$ on $\partial V(\delta)$ and the relation  $|(\bv_\Delta)_x-v_x|\le|(\bv_\Delta)_x-\bv_x|+|\bv_x-\hat{v}_x|+|\hat{v}_x-v_x|$ that we have 
$$\sup_{B(\Delta)}|(\bv_\Delta)_x-v_x|\le \epsilon_{18}(\Delta).$$ 
Hence, (\ref{fgfg1}) and (\ref{fgfg2}) imply that (\ref{f-g}) holds in all the cases (i)--(iii). Note that we fail to have (\ref{f-g}), if we simply follow the proof of Proposition \ref{key} with $\tilde{U}^{i^\ast}(\delta)$ in the definition of $w^k_{m+1}$.  We now do not have the estimate $|w^0_{m+1}-v^0_{m+1}|<\epsilon_{19}(\Delta)$ on $\mathcal{G}_{odd}|_{k=0}$. However, due to the choice of $\hat{v}$ and the normalization of $\hat{v},\bv,\bv_\Delta$ at $\gamma^\ast_{i^\ast}(0)$, it follows from the relation $\bv(x,0)-\hat{v}(x,0)=\int_{\gamma^\ast_{i^\ast}(0)}^x(\bv_x(y,0)-\hat{v}_x(y,0))dy$ that we have 
\[ {\rm (i)} \Rightarrow 
\bv(x,0)-\hat{v}(x,0)= \left\{ \begin{array}{ll}
\mbox{$0$\,\,\, for $\gamma^\ast_{i^\ast}(0)\le x\le\gamma^\ast_{i^\ast}(0)+\delta$},\\
 \mbox{negative\,\,\,  for $\gamma^\ast_{i^\ast-1}(0)\le x<\gamma^\ast_{i^\ast}(0)$},\qquad\qquad\qquad\qquad
\end{array} \right.
\]
\[ {\rm (ii)} \Rightarrow 
\bv(x,0)-\hat{v}(x,0)= \left\{ \begin{array}{ll}
\mbox{$0$\,\,\, for $\gamma^\ast_{i^\ast}(0)-\delta\le x\le\gamma^\ast_{i^\ast}(0)$},\\
 \mbox{negative\,\,\,  for $\gamma^\ast_{i^\ast}(0)< x\le\gamma^\ast_{i^\ast+1}(0)$},\qquad\qquad\qquad\qquad\,\,
\end{array} \right.
\]
\[ {\rm (iii)} \Rightarrow 
\bv(x,0)-\hat{v}(x,0)= \left\{ \begin{array}{ll}
\mbox{$0$\,\,\, for $x=\gamma^\ast_{i^\ast}(0)$},\\
 \mbox{negative\,\,\,  for $\gamma^\ast_{i^\ast-1}(0)\le x\le\gamma^\ast_{i^\ast+1}(0)$ and $x\neq\gamma^\ast_{i^\ast}(0)$}.\,\,\,\,\,\,
\end{array} \right.
\]
\noindent Since $\bv_\Delta\to\bv$ uniformly as $\Delta\to0$ and 
\begin{eqnarray*}
w^0_{m+1}-v^0_{m+1}&=&\bv^0_{m+1}-\hat{v}(x_{m+1},0)\\
&\le& \epsilon_{19}(\Delta)+\bv(x_{m+1},0)-\hat{v}(x_{m+1},0) \mbox{\,\,\,\, on $\tilde{V}(\delta)$},\\
w^0_{m+1}-v^0_{m+1}&=&\bv^0_{m^\ast(0)+1}-v_{m^\ast(0)+1}^0\\
&\le& \epsilon_{19}(\Delta)+\bv(x_{m^\ast(0)+1},0)-\hat{v}(x_{m^\ast(0)+1},0) \mbox{\,\,\,\, for ${m^\ast(k)}<{m}$},\\
w^0_{m+1}-v^0_{m+1}&=&\bv^0_{m_\ast(0)+1}-v_{m_\ast(0)+1}^0\\
&\le& \epsilon_{19}(\Delta)+\bv(x_{m_\ast(0)+1},0)-\hat{v}(x_{m_\ast(0)+1},0) \mbox{\,\,\,\, for ${m}<{m_\ast(k)}$},
\end{eqnarray*}
we have in all the cases (i)--(iii),
\begin{eqnarray*}
w^0_{m+1}-v^0_{m+1}<\epsilon_{19}(\Delta)\mbox{\,\,\, on $\mathcal{G}_{odd}|_{k=0}$}.
\end{eqnarray*}
Take $n$ so that $w^0_{n+1}-v^0_{n+1}$ is equal to  
$$\max_{\tilde{U}^{i^\ast}(\delta/3)|_{k=0}}(w^0_{m+1}-v^0_{m+1}).$$
Here, ``$\max_{\tilde{U}^{i^\ast}(\delta/3)|_{k=0}}$'' is not ``$\max_{\tilde{V}(\delta/3)|_{k=0}}$''. Note that 
$$  -\epsilon_{20}(\Delta)<w^0_{n+1}-v^0_{n+1},$$
 because 
\begin{eqnarray*}
w^0_{n+1}-v^0_{n+1}&=&\bv_\Delta(x_{n+1},0)-\bv(x_{n+1},0)+\bv(x_{n+1},0)-\hat{v}(x_{n+1},0)\\
&\ge& \bv_\Delta(x_{m(\gamma^\ast_{i^\ast}(0))},0)-\bv(x_{m(\gamma^\ast_{i^\ast}(0))},0)+\bv(x_{m(\gamma^\ast_{i^\ast}(0))},0)-\hat{v}(x_{m(\gamma^\ast_{i^\ast}(0))},0),
\end{eqnarray*}
$\bv_\Delta-\bv\to0$ on $\tilde{U}^{i^\ast}(\delta/3)|_{k=0}$ and $\bv(\gamma^\ast_{i^\ast}(0),0)-\hat{v}(\gamma^\ast_{i^\ast}(0),0)=0$.  
Then, we can estimate $w^0_{n+1}-v^0_{n+1}$ from the above with the minimizer $\xi^\ast$ for $v^0_{n+1}$ in the same manner as that of the proof of Proposition \ref{key}, obtaining 
\begin{eqnarray}\label{goal}
(h(c )-h_\Delta(c ))qT\le T\Gamma^{i^\ast}\frac{\dx}{2\lambda}+\nu(\delta)\frac{\dx}{2\lambda}+\frac{\epsilon_{21}(\Delta)}{\delta^2}\dx+\epsilon_{22}(\Delta)\dx. 
\end{eqnarray} 
On the other hand, by the argument in the proof of Proposition \ref{key}, we have for $\gamma^\ast_i$ ($i\neq i^\ast$) at which $\bv$ has transition points, 
$$(h(c )-h_\Delta(c ))qT\ge T\Gamma^{i}\frac{\dx}{2\lambda}-\nu(\delta)\frac{\dx}{2\lambda}-\frac{\epsilon_{23}(\Delta)}{\delta^2}\dx-\epsilon_{24}(\Delta)\dx.$$  
Thus, by (A3), we reach a contradiction as $\Delta\to0$.      
\end{proof}
We extend Proposition \ref{existence}: 
\begin{Prop}\label{extension}
For each $c\in(c_0,c_1)$, the limit of any convergent subsequence $\{\bv^{(c )}_\Delta\}$ has transition points at $\gamma^\ast_{i^\ast}$. 
\end{Prop}
\begin{proof}
Let $c^\ast:=\max J^-\cap J^+$. Then, Proposition \ref{existence}  implies that the limit $\bv^{(c^\ast)}$ of any convergent subsequence $\{\bv^{(c^\ast)}_\Delta\}$ has transition points at $\gamma^\ast_{i^\ast}$. By (A3) and Proposition \ref{key}, $\bv^{(c^\ast)}$  has transition points only at $\gamma^\ast_{i^\ast}$. It follows from Subsection 3.3 and the choice of $c^\ast$that we may have the $\Z^2$-periodic viscosity solution $\hat{v}^{(c^\ast)}$ such that  graph$(c^\ast+\hat{v}^{(c^\ast)}_x)|_{t=0}\supset\cup_{0\le k<q}(\mathcal{U}^-_k\cup\mathcal{U}^+_k)$ and graph$(c^\ast+\hat{v}^{(c^\ast)}_x)|_{t=0}\setminus \cup_{0\le k<q} \mathcal{U}^-_k$ is a part of the upper separatrix at $t=0$. It is clear that $\hat{v}^{(c^\ast)}$ has transition points only at $\gamma^\ast_{i^\ast}$.  By Proposition \ref{peierls}, $\bv^{(c^\ast)}$ and $\hat{v}^{(c^\ast)}$ must coincide (up to constants) with the Peierls barrier. Therefore we see that  
$$\mbox{graph$(c^\ast+\bar{v}^{(c^\ast)}_x)|_{t=0}=$graph$(c^\ast+\hat{v}^{(c^\ast)}_x)|_{t=0}$}.$$  
Suppose that there exists $c\in (c^\ast,c_1)$ for which ($\ast$) does not hold:

 ($\ast$)\quad  The limit of any convergent subsequence $\{\bv^{(c )}_\Delta\}$ has transition points at $\gamma^\ast_{i^\ast}$. 

\noindent Then, we have a convergent subsequence $\{\bvc_\Delta\}$ whose limit $\bvc$ does not have transition points at $\gamma^\ast_{i^\ast}$. It follows from (2) of Proposition \ref{preliminary} that we have $c^\ast+(\bv^{(c^\ast)}_\Delta)_x<c+(\bvc_\Delta)_x$ a.e., obtaining  $c^\ast+\bv^{(c^\ast)}_x|_{t=0}\le c+\bvc_x|_{t=0}$ a.e. Hence, if graph$(c+\bvc_x)|_{t=0}$ does not contain $\mathcal{S}^-_k$, then $\bvc$ must have transition points at $\gamma^\ast_{i^\ast}$. Therefore, we see that graph$(c+\bvc_x)|_{t=0}\supset\cup_{0\le k<q}\mathcal{S}^-_k$. However, this means that graph$(c+\bvc_x)|_{t=0}$ coincides with the whole upper separatrix at $t=0$, which is only allowed for $c=c_1$. Therefore, ($\ast$) holds for all $c\in(c^\ast,c_1)$.  

 We have the lower extension with $c_\ast:=\min J^-\cap J^+$ in the same manner. 
\end{proof}

\begin{proof}[Proof of Theorem \ref{main}.] 
By Proposition \ref{extension}, the limit $\bvc$ of any convergent subsequence must have transition points at $\gamma^\ast_{i^\ast}$. By (A3) and Proposition \ref{key}, $\bvc$ has transition points only at $\gamma^\ast_{i^\ast}$. By Proposition \ref{peierls}, such $\bvc$ is unique up to constants and coincides with $h^{(c )}_p(\gamma^\ast_{i^\ast}(0),0;\cdot,\cdot)$.  Therefore, with the normalization, we conclude that $\bvc_\Delta\to h^{(c )}_p(\gamma^\ast_{i^\ast}(0),0;\cdot,\cdot)$ as $\Delta\to0$ in the whole sequence. By (5) of Proposition \ref{preliminary}, this  implies that $\buc_\Delta\to (h^{(c )}_p(\gamma^\ast_{i^\ast}(0),0;\cdot,\cdot))_x$ as $\Delta\to0$ in the whole sequence. 
\end{proof}

\end{document}